\documentclass[12pt]{amsart}  

\usepackage[utf8]{inputenc}
\usepackage{amsmath} 
\usepackage{amsfonts}
\usepackage{amssymb}
\usepackage{stmaryrd}
\usepackage{latexsym} 
\usepackage{graphicx}
\usepackage{subfigure}
\usepackage{color}
\usepackage{hyperref}
\usepackage{verbatim}
\usepackage[all]{xy}
\usepackage{graphics}
\usepackage{pdfsync}
\usepackage{tikz}
\usepackage{url}
\usepackage{enumerate}
\usetikzlibrary{quotes}

\theoremstyle{definition}
\newtheorem{theorem}{Theorem}[section]
\newtheorem*{theorem*}{Theorem}
\newtheorem{proposition}[theorem]{Proposition}

\newtheorem{lemma}[theorem]{Lemma}
\newtheorem{corollary}[theorem]{Corollary}
\newtheorem{conjecture}[theorem]{Conjecture}
\newtheorem{definition}[theorem]{Definition}
\newtheorem{example}[theorem]{Example}

\theoremstyle{remark}
\newtheorem{remark}[theorem]{Remark}

\numberwithin{equation}{section}
\newcommand{\Sym}{\ensuremath{\operatorname{Sym}}}

\newcommand{\eps}{\varepsilon}
\newcommand{\ups}{\upsilon}
\newcommand{\p}{P}
\newcommand{\dju}{\oplus} 
\newcommand{\grad}{\nabla}
\newcommand{\ndiv}{\partial}
\newcommand{\sgn}{\operatorname{sgn}}
\newcommand{\nL}{\mathcal L}
\newcommand{\cut}{\operatorname{Cut}}
\newcommand{\cyc}{\operatorname{Cyc}}
\newcommand{\ex}{\mathbb E}
\newcommand{\ind}{1}
\newcommand{\dire}{\vec} 
\newcommand{\g}{g}
\newcommand{\indb}{\mathbf 1}
\DeclareMathOperator{\im}{im}

\newlength\cellsize 
\savebox2{%
\begin{picture}(15,15)
\put(0,0){\line(1,0){15}}
\put(0,0){\line(0,1){15}}
\put(15,0){\line(0,1){15}}
\put(0,15){\line(1,0){15}}
\end{picture}}
\newcommand\cellify[1]{\def\thearg{#1}\def\nothing{}%
\ifx\thearg\nothing
\vrule width0pt height\cellsize depth0pt\else
\hbox to 0pt{\usebox2\hss}\fi%
\vbox to 15\unitlength{
\vss
\hbox to 15\unitlength{\hss$#1$\hss}
\vss}}
\newcommand\tableau[1]{\vtop{\let\\=\cr
\setlength\baselineskip{-16000pt}
\setlength\lineskiplimit{16000pt}
\setlength\lineskip{0pt}
\halign{&\cellify{##}\cr#1\crcr}}}
\savebox3{%
\begin{picture}(15,15)
\put(0,0){\line(1,0){15}}
\put(0,0){\line(0,1){15}}
\put(15,0){\line(0,1){15}}
\put(0,15){\line(1,0){15}}
\end{picture}}
\newcommand\expath[1]{%
\hbox to 0pt{\usebox3\hss}%
\vbox to 15\unitlength{
\vss
\hbox to 15\unitlength{\hss$#1$\hss}
\vss}}
\newcommand\bas[1]{\omit \vbox to \cellsize{ \vss \hbox to \cellsize{\hss$#1$\hss} \vss}}

\begin{document}

\title[Deletion-contraction for a unified Laplacian and applications]{Deletion-contraction for a unified Laplacian\\ and applications}

\author{Farid Aliniaeifard}
\address{
 Department of Mathematics,
 University of British Columbia,
 Vancouver BC V6T 1Z2, Canada}
\email{farid@math.ubc.ca}

\author{Victor Wang}
\address{
 Department of Mathematics,
 University of British Columbia,
 Vancouver BC V6T 1Z2, Canada}
\email{vyzwang@student.ubc.ca}

\author{Stephanie van Willigenburg}
\address{
 Department of Mathematics,
 University of British Columbia,
 Vancouver BC V6T 1Z2, Canada}
\email{steph@math.ubc.ca}

\thanks{
All authors were supported  in part by the National Sciences and Engineering Research Council of Canada.}

\begin{abstract} 
We define a graph Laplacian with vertex weights in addition to the more classical edge weights, which unifies the combinatorial Laplacian and the normalised Laplacian. Moreover, we give a combinatorial interpretation for the coefficients of the weighted Laplacian characteristic polynomial in terms of weighted spanning forests and use this to prove a deletion-contraction relation. We prove various interlacing theorems relating to deletion and contraction, as well as to rectangular tilings, drawing on the work of Brooks, Smith, Stone and Tutte on square tilings. Additionally, we show that the weighted Laplacian also satisfies a vertex analogue of deletion-contraction. We give applications of weighted Laplacian eigenvalues to sparse cuts, independent sets and graph colouring, and establish new cases of a conjecture of Stanley on distinguishing nonisomorphic trees.
\end{abstract}

\maketitle
\tableofcontents

\section{Introduction}\label{sec:intro}  
The combinatorial Laplacian matrix $L_G=D_G-A_G$ of a graph $G$, defined by the difference between the degree and adjacency matrices of $G$, was perhaps first studied by Kirchhoff \cite{Kir} in his 1847 paper on electrical networks. In the same paper, Kirchhoff discovered the celebrated matrix-tree theorem, which states that the number of spanning trees of a graph $G$ is equal to any cofactor of $L_G$. Intriguingly, the matrix-tree theorem also arises in the work of Brooks, Smith, Stone and Tutte \cite{BSST} on square tilings, which they also found were related to electrical circuits.

In the 1990's, Chung \cite{ChungBook} popularised the normalised Laplacian matrix, sometimes defined by $\nL_G=D_G^{-1}L_G$, which has connections to random walks as $\nL_G = I-D_G^{-1}A_G$ where $D_G^{-1}A_G$ is the random walk matrix. The combinatorial and normalised Laplacians may be viewed as discrete analogues of the (negative) Laplace-Beltrami operator on Riemannian manifolds, defined by the negative divergence of the gradient, and this connection has been used to relate the second combinatorial and normalised Laplacian eigenvalues of a graph to sparse (normalised) cuts \cite{ChungBook, Dodziuk}, giving a discrete analogue of Cheeger's isoperimetric inequality in Riemannian geometry \cite{Cheeger}. In addition, combinatorial and normalised Laplacian eigenvalues can be used to obtain bounds on various parameters of a graph, including its chromatic number \cite{ChungBook,CGP}, diameter \cite{ChungBook,Mohar} and independence number \cite{GN}.

Another topic of interest in graph theory is the notion of deletion-contraction. Many graph polynomials and invariants, including number of acyclic orientations \cite{Stan73}, spanning trees and spanning forests, as well as a graph's chromatic and flow polynomials \cite{TutChrom}, satisfy a recursive formula relating a graph and the two graphs obtained by deleting and contracting an edge. The combinatorial and normalised Laplacian characteristic polynomials, however, are not known to satisfy a deletion-contraction recursion. The effects of edge deletion, contraction and related operations on combinatorial and normalised Laplacian eigenvalues are studied across \cite{Butler, CDHLPS, GMS}.

The definitions of the combinatorial and normalised Laplacians can be naturally extended to edge-weighted graphs $(G,\eps)$, and these have been studied by several authors, including by Kirchhoff in his original paper \cite{Kir}, where the edge weights represent the conductances of wires. In contrast, graph Laplacians with vertex weights are not so commonly studied, and there exist competing definitions for such a matrix \cite{CL, FT, Lovasz}. Vertex-weighted graphs appear in the work of Noble and Welsh \cite{NW}, in which they define a polynomial associated to integer vertex-weighted graphs satisfying a deletion-contraction rule, motivated by the study of Vassiliev invariants of knots \cite{CDL}. Vertex weights also arise in graph theory applications; when studying independent sets in job scheduling, the weight of a vertex can be used to represent the relative reward for the completion of each job.

Our paper defines a weighted Laplacian for graphs with both vertex and edge weights, which specialises to the combinatorial and normalised Laplacians for certain choices of vertex weights, unifying the discussion of the two Laplacian matrices. Moreover, the incorporation of vertex weights gives rise to a deletion-contraction recursion for the weighted Laplacian characteristic polynomial. Our paper is more precisely structured as follows. 

We cover the necessary background in Section~\ref{sec:bg}. In Section~\ref{sec:calc} we introduce the inner product spaces $L^2(V,\ups)$ and $L^2(E,\eps)$ of scalar fields and vector fields, respectively, associated with a weighted graph $(G,\ups,\eps)$, and define discrete analogues of negative divergence and gradient. We use these to define the weighted Laplacian $L_{(G,\ups,\eps)}$ and edge Laplacian $K_{(G,\ups,\eps)}$ in Definitions~\ref{def:wlap} and \ref{def:elap}, and briefly develop their theory. Then in Section~\ref{sec:delcon} we give a combinatorial interpretation of weighted Laplacian characteristic polynomial coefficients in terms of weighted spanning forests in Theorem~\ref{the:mforest}, and deduce a deletion-contraction recurrence in Theorem~\ref{the:delcon}. In Section~\ref{sec:interlace} we use deletion-contraction to obtain an interlacing theorem relating the weighted Laplacian eigenvalues of a weighted graph and those obtained by deleting and contracting an edge in Theorem~\ref{the:conlace}. We then apply our techniques to obtain various interlacing results on combinatorial and normalised Laplacian eigenvalues of quotient graphs in Corollaries~\ref{cor:cquot} and \ref{cor:nquot}, as well as graphs obtained by subgraph deletion in Corollaries~\ref{cor:subgraph} and \ref{cor:nsubgraph}. In Section~\ref{sec:elec} we study interlacing theorems arising from rectangular tilings and Kron reduction in Theorems~\ref{the:rect} and \ref{the:kron}, and prove a vertex analogue of deletion-contraction in Theorem~\ref{the:addred}, which we call addition-reduction. We give applications of weighted Laplacian eigenvalues in Section~\ref{sec:bounds}, proving an isoperimetric inequality in Corollary~\ref{cor:cheeger} and bounds on independent sets and the chromatic number in Theorems~\ref{the:indep} and \ref{the:chrom}. Finally, in Section~\ref{sec:csf} we use deletion-contraction to relate the weighted Laplacian characteristic polynomial to the chromatic symmetric function and other graph polynomials, and use this to establish new cases of Stanley's tree isomorphism conjecture in Corollary~\ref{cor:newtrees}.

\section{Background}\label{sec:bg}
A \textit{graph} $G$ consists of a vertex set $V(G)$ and an edge set $E(G)$ representing connections between pairs of vertices. When the underlying graph $G$ is clear, we may denote the vertex and edge sets of a graph simply by $V$ and $E$, respectively. All graphs in this paper will be \textit{finite multigraphs}, with finite vertex and edge sets, the latter of which may include \textit{loops} (edges connecting some vertex to itself) or \textit{multiple edges} (two or more edges incident to the same pair of vertices). A graph is \textit{simple} if it has no loops or multiple edges.

For $u,v\in V$, we write $uv$ to mean an edge connecting $u$ and $v$. A \textit{directed edge} consists of an edge together with a choice of one of two orientations; an edge $uv$ may either be oriented from $u$ to $v$, with corresponding directed edge written $\vec {uv}$, or from $v$ to $u$, with corresponding directed edge written $\vec {vu}$. The set of directed edges of a graph $G$ is denoted $\dire E(G)$, or $\dire E$ when the underlying graph $G$ is clear.

If $e$ is an edge not in $E$, we write $G+e$ to denote the graph obtained by adding $e$ to the edge set of $G$. We write $G-S$ for $S\subseteq V$ to denote the graph obtained from $G$ by removing all vertices in $S$ and any incident edges. Similarly, for $R\subseteq E$, we write $G-R$ to denote the graph obtained from $G$ by removing the edges in $R$. When $R=\{e\}$, we may also denote this by $G-e$. Additionally, we write $G[R]$ for $R\subseteq E$ to denote the subgraph of $G$ \textit{induced} by $R$, consisting of the edges in $R$ and all incident vertices.

A \textit{weighted graph} $(G,\ups,\eps)$ consists of a graph $G$ together with \textit{vertex weights} and \textit{edge weights} given by weight functions $\ups:V\to \mathbb R_{>0}$ and $\eps:E\to\mathbb R_{>0}$, respectively. An \textit{edge-weighted graph} $(G,\eps)$ is a weighted graph in which vertex weights are negligible and taken to equal $1$. Similarly, a \textit{vertex-weighted graph} $(G,\ups)$ is a weighted graph in which edge weights are negligible and taken to equal $1$. We may also write $G$ to denote an \textit{unweighted graph}, in which all vertex and edge weights are taken to equal $1$.

If $G$ and $G'$ are two graphs, we write $G\dju G'$ to denote their disjoint union. Similarly, if $(G,\ups,\eps)$ and $(G',\ups',\eps')$ are two weighted graphs, we denote their disjoint union by $(G,\ups,\eps)\dju(G',\ups',\eps')$.

Given a graph $G$ and an equivalence relation $\sim$ on its vertex set $V$, the \textit{quotient graph} $G/\sim$ is the graph obtained by formally identifying the vertices in each equivalence class of $\sim$. If $(G,\ups,\eps)$ is a weighted graph, the \textit{quotient weighted graph} $(G,\ups,\eps)/\sim$ is $(G/\sim,\ups/\sim,\eps)$, where the weight under $\ups/\sim$ of an element of $V/\sim$ is the sum of the weights under $\ups$ of the vertices in the equivalence class of $V$. For $S\subseteq V$, we write $S^c$ to denote its \textit{complement} $V\setminus S$.

Given an edge-weighted graph $(G,\eps)$ and two vertices $u,v$, we write $\eps(u,v)$ to denote the total edge weight between $u$ and $v$, with loops counted twice. The \textit{degree} $d_{(G,\eps)}(u)$ of a vertex $u$ is given by the sum $\sum_{v\in V}\eps(u,v)$. An \textit{isolated vertex} is a vertex of degree zero. If $(G,\eps)$ is an edge-weighted graph with vertices arbitrarily ordered as $v_1,\dots,v_n$, the \textit{adjacency matrix} of $(G,\eps)$ is the $n\times n$ matrix $(a_{ij})$, where each $a_{ij}=\eps(v_i,v_j)$, and the \textit{degree matrix} $D_{(G,\eps)}$ is the $n\times n$ diagonal matrix with $i$th diagonal entry given by $d_{(G,\eps)}(v_i)$.

The study of spectral graph theory relates the eigenvalues of matrices associated to graphs and weighted graphs to graph parameters. Two important matrices associated to an edge-weighted graph are the combinatorial Laplacian and the normalised Laplacian.

\begin{definition}
The \textit{combinatorial Laplacian} of an edge-weighted multigraph $(G,\eps)$ with vertex set $\{v_1,\dots,v_n\}$ is
$$L_{(G,\eps)}= D_{(G,\eps)}-A_{(G,\eps)}.$$
\end{definition}

One application of the combinatorial Laplacian is the celebrated matrix-tree theorem. We state a version on principal minors of the combinatorial Laplacian for edge-weighted graphs, which may be deduced from the all minors matrix-tree theorem in \cite{Chaiken}. We will see that each principal minor of $L_{(G,\eps)}$ has a combinatorial interpretation in terms of \textit{rooted spanning forests}, which consist of a spanning forest of $G$, together with a choice of a vertex from each connected component of the spanning forest, constituting the \textit{roots}. A spanning forest is \textit{$k$-rooted} if it has $k$ roots, and \textit{$S$-rooted} for $S\subseteq V$ if its roots are given by $S$.

\begin{theorem}[Principal minors matrix-tree theorem]
\label{the:mtree}
Let $(G,\eps)$ be an edge-weighted multigraph and let $S\subseteq V$. Let $L_{(G,\eps)}^{\hat S}$ denote the principal submatrix of $L_{(G,\eps)}$ obtained by deleting the rows and columns corresponding to $S$. Then
$$\det L^{\hat S}_{(G,\eps)} = \sum_F \prod_{e\in E(F)}\eps(e),$$
where the sum is over all $S$-rooted spanning forests $F$ of $G$.
\end{theorem}

\begin{definition}
The \textit{normalised Laplacian} of an edge-weighted multigraph $(G,\eps)$ with no isolated vertices and vertex set $\{v_1,\dots,v_n\}$ is
$$\nL_{(G,\eps)}=D_{(G,\eps)}^{-1}L_{(G,\eps)}.$$
\end{definition}

Many authors define the normalised Laplacian instead by $D_{(G,\eps)}^{-1/2}L_{(G,\eps)}D_{(G,\eps)}^{-1/2}$, which is a real symmetric matrix. However, the two possible definitions of the normalised Laplacian give similar matrices, as $D_{(G,\eps)}^{-1/2}L_{(G,\eps)}D_{(G,\eps)}^{-1/2}=D_{(G,\eps)}^{1/2}\nL_{(G,\eps)}D_{(G,\eps)}^{-1/2}$. In particular, the two matrices share the same eigenvalues, and the eigenvectors of one may be recovered from the eigenvectors of the other by the change of basis matrix $D_{(G,\eps)}^{1/2}$. The motivation for having the normalised Laplacian be symmetric is given by the spectral theorem on self-adjoint operators.

If $A:V\to W$ is a linear map between two finite-dimensional inner product spaces, then its \textit{adjoint} is the unique linear map $A^*:W\to V$ satisfying $\langle Ax,y\rangle = \langle x,A^*y\rangle$ for all $x\in V$ and $y\in W$. A linear map $A:V\to V$ is \textit{self-adjoint} if it is equal to its own adjoint. Note that when the inner product space $V$ is given by $\mathbb R^n$ with the Euclidean inner product, the self-adjoint maps $V\to V$ are exactly the $n\times n$ real symmetric matrices.

\begin{theorem}[Spectral theorem]
    Let $V$ be an $n$-dimensional inner product space, and let $A:V\to V$ be a self-adjoint linear map. Then $A$ has $n$ real eigenvalues, and there exists an orthonormal basis of $V$ consisting of eigenvectors of $A$.
\end{theorem}

In particular, for any edge-weighted graph $(G,\eps)$ on $n$ vertices, $L_{(G,\eps)}$ has $n$ real eigenvalues and there exists an orthonormal basis of $\mathbb R^n$ consisting of eigenvectors of $L_{(G,\eps)}$. Since $\nL_{(G,\eps)}$ is similar to symmetric $D_{(G,\eps)}^{-1/2}L_{(G,\eps)}D_{(G,\eps)}^{-1/2}$, it also has $n$ real eigenvalues.

Given a self-adjoint linear map $A:V\to V$, where $V$ is a finite-dimensional inner product space, the \textit{Rayleigh quotient} of nonzero $x\in V$ is the value $\frac{\langle Ax,x\rangle}{\langle x,x\rangle}$. Note when $x$ is an eigenvector of $A$, the Rayleigh quotient of $x$ is the corresponding eigenvalue. Rayleigh quotients can be used to bound the extreme eigenvalues of a self-adjoint operator.

\begin{lemma} \label{lem:rq}
    Let $V$ be an $n$-dimensional inner product space, and let $A:V\to V$ be a self-adjoint linear map. Let $\lambda_1$ and $\lambda_n$ denote the smallest and largest eigenvalues of $A$, respectively. Then for any nonzero $x\in V$, we have
$$\lambda_1\le \frac{\langle Ax,x\rangle}{\langle x,x\rangle} \le \lambda_n.$$
\end{lemma}

An \textit{isometry} $\iota: V\to W$ between finite-dimensional inner product spaces $V$ and $W$ is a linear map satisfying $\langle x,y\rangle=\langle\iota x,\iota y\rangle$ for all $x,y\in V$. Note that an isometry is necessarily injective.

\begin{theorem}[Cauchy interlacing theorem] \label{the:cauchy}
Let $V$ and $W$ be finite-dimensional inner product spaces, and let $\iota:V\to W$ be an isometry. Let $\lambda_1\le\dots\le\lambda_n$ and $\mu_1\le\dots\le\mu_{n-k}$ denote the eigenvalues of $A$ and $\iota^*A\iota$, respectively, where $A:W\to W$ is a self-adjoint linear map. Then
$$\lambda_i\le\mu_i\le\lambda_{i+k}$$
for all $1\le i \le n-k$.
\end{theorem}
\begin{remark}\label{rem:preserve}
Under the hypotheses of the theorem, in order for a linear map $B:V\to V$ to be equal to $\iota^*A\iota$, it is sufficient that $B$ be self-adjoint and preserve quadratic form. To see this, note if
$$\langle Bx,x\rangle = \langle A\iota x, \iota x\rangle$$ for all $x\in V$, then $B-\iota^*A\iota$ is self-adjoint and satisfies $\langle(B-\iota^*A\iota)x,x\rangle = 0$ for all $x\in V$, implying that $B-\iota^*A\iota$ is the zero matrix, having all eigenvalues zero.
\end{remark}

One common application of the Cauchy interlacing theorem is when $V$ and $W$ are given by $\mathbb R^{n-k}$ and $\mathbb R^n$ with the usual Euclidean inner product. If $A$ is an $n\times n$ real symmetric matrix and $B$ is an $(n-k)\times(n-k)$ principal submatrix, then there is an isometry $\iota:\mathbb R^{n-k}\to\mathbb R^n$ for which $B=\iota^*A\iota$. Note that the Cauchy interlacing theorem also applies when taking a principal submatrix of the matrix of a self-adjoint map in an orthogonal, and not necessarily orthonormal, basis.

Finally, we will also write $[n]$ for $n\in\mathbb Z_{>0}$ to denote the set $\{1,\dots,n\}$.

\section{Discrete vector calculus and the weighted Laplacian}\label{sec:calc}
In this section, we will introduce important inner product spaces and linear maps associated with a weighted graph $(G,\ups,\eps)$, and describe analogies to concepts from classical vector calculus.

Since $\ups$ is a map $V\to\mathbb R_{>0}$, it can be extended to define a measure on $V$. We define $L^2(V,\ups)$ to be the real inner product space of \textit{scalar fields} $f:V\to\mathbb R$, with inner product given by
$$\langle f,g\rangle = \int_V fgd\ups = \sum_{v\in V}f(v)g(v)\ups(v).$$

Similarly, $\eps: E\to\mathbb R_{>0}$ can also be extended to define a measure on $E$. We next define the real inner product space $L^2(E,\eps)$ of \textit{vector fields} on our graph. If we make an arbitrary choice of orientation for each edge of our graph, then to a first approximation we can think of vector fields $\mathbf F\in L^2(E,\eps)$ as real-valued functions $\mathbf F: E\to\mathbb R$. Then the definition of $L^2(E,\eps)$ and its inner product coincides with the usual definitions for the $L^2$ space associated with a measure space.

However, it is more natural to think of elements of $L^2(E,\eps)$ as vector-valued functions on $E$. To each edge $uv\in E$, we associate a unique real inner product space spanned by the unit vector
$$\mathbf v_{\vec {uv}} = -\mathbf v_{\vec{vu}},$$
where we think of $\mathbf v_{\vec{uv}}$ as a vector pointing towards $v$ and away from $u$, and of $\mathbf v_{\vec{vu}}$ as a vector pointing towards $u$ and away from $v$. Then a vector field $\mathbf F\in L^2(E,\eps)$ is a choice for each edge $e\in E$ of a vector $\mathbf F(e)$ in the space associated to $e$, and the inner product on $L^2(E,\eps)$ is given by
$$\langle\mathbf F,\mathbf G\rangle = \int_E \mathbf F\cdot\mathbf G d\eps = \sum_{e\in E}\mathbf F(e)\cdot\mathbf G(e)\eps(e).$$

We will next introduce two important linear maps between the spaces defined. The first map is defined by
\begin{align*}
\ndiv: L^2(E,\eps)&\to L^2(V,\ups)\\
\ndiv\mathbf F(v) &= \sum_{\substack{uv\in E\\ u\neq v}}\frac{\eps(uv)}{\ups(v)}\mathbf F(uv)\cdot \mathbf v_{\vec{uv}}.
\end{align*}
Applying $\ndiv$ to a vector field gives a measure of the net flow into each vertex per unit vertex weight, so $\ndiv$ is the discrete analogue of negative divergence from vector calculus.

\begin{example}\label{ex:ndiv}
Let $(G,\ups,\eps)$ be the weighted graph depicted on the left, and let $\mathbf F\in L^2(E,\eps)$ be the vector field depicted in the centre. Then $\ndiv\mathbf F\in L^2(V,\ups)$ is the scalar field depicted on the right. The value of $\ndiv\mathbf F$ at, for example, the vertex of weight $2$ in $(G,\ups,\eps)$ is obtained by computing $\frac{-2-6-2+3}{2}=-\frac{7}{2}$.
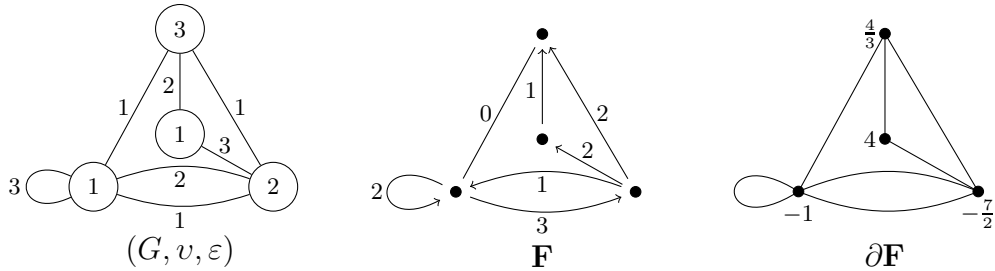
\begin{center}
\begin{tikzpicture}
\node[shape=circle, draw=black](1) at (4.75,-1.05) {\scriptsize 1};
\node[shape=circle, draw=black](2) at (7.14,-1.05) {\scriptsize 2};
\node[shape=circle,  draw=black] (3) at (5.9,-.35) {\scriptsize 1};
\node[shape=circle,  draw=black] (4) at (5.9,1.05) {\scriptsize 3};
\path[-] (1) edge (4);
\path[-] (3) edge (4);
\path[-] (2) edge (4);
\path[-] (3) edge (2);
\path[-, bend left=18.75] (1) edge (2);
\path[-, bend right=18.75] (1) edge (2);
\draw[-] (1)  to[in=205,out=155,loop] (1);
\node[] at (5.15,0) {\scriptsize 1};
\node[] at (6.7,0) {\scriptsize 1};
\node[] at (5.75,.3) {\scriptsize 2};
\node[] at (6.5,-.5) {\scriptsize 3};
\node[] at (5.9,-.95) {\scriptsize 2};
\node[] at (5.9,-1.5) {\scriptsize 1};
\node[] at (3.7,-1.05) {\scriptsize 3};
\node[] at (5.9,-1.9) { $(G,\ups,\eps)$};
\end{tikzpicture}\quad\quad
\begin{tikzpicture}
\node[shape=circle](1) at (4.75,-1.05) {};
\filldraw [black] (1) circle (2pt);
\node[shape=circle](2) at (7.14,-1.05) {};
\filldraw [black] (2) circle (2pt);
\node[shape=circle] (3) at (5.9,-.35) {};
\filldraw [black] (3) circle (2pt);
\node[shape=circle] (4) at (5.9,1.05) {};
\filldraw [black] (4) circle (2pt);
\path[-] (1) edge (4);
\path[->] (3) edge (4);
\path[->] (2) edge (4);
\path[<-] (3) edge (2);
\path[<-, bend left=20] (1) edge (2);
\path[->, bend right=20] (1) edge (2);
\path[->, scale=2.2] (1) edge[in=207.5,out=152.5,loop] (1);
\node[] at (5.15,0) {\scriptsize 0};
\node[] at (6.7,0) {\scriptsize 2};
\node[] at (5.75,.3) {\scriptsize 1};
\node[] at (6.5,-.5) {\scriptsize 2};
\node[] at (5.9,-.95) {\scriptsize 1};
\node[] at (5.9,-1.5) {\scriptsize 3};
\node[] at (3.7,-1.05) {\scriptsize 2};
\node[] at (5.9,-1.9) { $\mathbf F$};
\end{tikzpicture}\quad\quad
\begin{tikzpicture}
\coordinate (1) at (4.75,-1.05);
\coordinate (2) at (7.14,-1.05);
\coordinate (3) at (5.9,-.35);
\coordinate (4) at (5.9,1.05);
\filldraw [black] (1) circle (2pt);
\filldraw [black] (2) circle (2pt);
\filldraw [black] (3) circle (2pt);
\filldraw [black] (4) circle (2pt);
\draw[-] (1) to [bend left=22.5] (2);
\draw[-] (1) to [bend right=22.5] (2);
\draw[scale=2.7] (1)  to[in=212.5,out=147.5,loop] (1);
\draw[-] (1)--(4);
\draw[-] (3)--(4);
\draw[-] (2)--(4);
\draw[-] (3)--(2);
\node[] at (5.7,1.05) {\scriptsize $\frac{4}{3}$};
\node[] at (4.75,-1.35) {\scriptsize $-1$};
\node[] at (7.14,-1.35) {\scriptsize $-\frac{7}{2}$};
\node[] at (5.7,-.35) {\scriptsize $4$};
\node[] at (5.9,-1.9) { $\ndiv\mathbf F$};
\end{tikzpicture}
\end{center}
\end{example}

The second map is defined by
\begin{align*}
\grad: L^2(V,\ups)&\to L^2(E,\eps)\\
\grad f(uv)&= (f(v)-f(u))\mathbf v_{\vec{uv}}.
\end{align*}
Applying $\grad$ to a scalar field gives at each edge a vector pointing in the direction of steepest ascent, so $\grad$ is the discrete analogue of gradient from vector calculus.

\begin{example}
Let $(G,\ups,\eps)$ be the weighted graph in Example~\ref{ex:ndiv}, and let $f\in L^2(V,\ups)$ be the scalar field depicted on the left. Then $\grad f\in L^2(E,\eps)$ is the vector field depicted on the right. The value of $\grad f$ on, for example, the edge between the vertex $u$ of weight $3$ and the vertex $v$ of weight $2$ in $(G,\ups,\eps)$ is obtained by computing $(1+2)\mathbf v_{\vec{uv}}=3\mathbf v_{\vec {uv}}.$
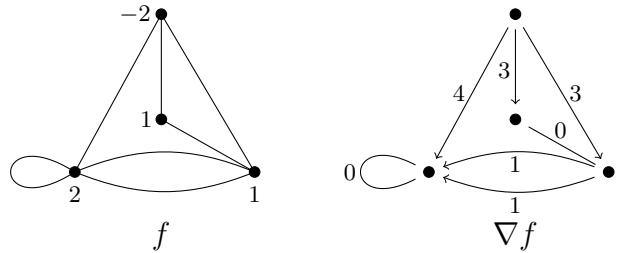
\begin{center}
\begin{tikzpicture}
\coordinate (1) at (4.75,-1.05);
\coordinate (2) at (7.14,-1.05);
\coordinate (3) at (5.9,-.35);
\coordinate (4) at (5.9,1.05);
\filldraw [black] (1) circle (2pt);
\filldraw [black] (2) circle (2pt);
\filldraw [black] (3) circle (2pt);
\filldraw [black] (4) circle (2pt);
\draw[-] (1) to [bend left=22.5] (2);
\draw[-] (1) to [bend right=22.5] (2);
\draw[scale=2.7] (1)  to[in=212.5,out=147.5,loop] (1);
\draw[-] (1)--(4);
\draw[-] (3)--(4);
\draw[-] (2)--(4);
\draw[-] (3)--(2);
\node[] at (5.55,1.05) {\scriptsize $-2$};
\node[] at (4.75,-1.35) {\scriptsize $2$};
\node[] at (7.14,-1.35) {\scriptsize $1$};
\node[] at (5.7,-.35) {\scriptsize $1$};
\node[] at (5.9,-1.9) { $f$};
\end{tikzpicture}\quad\quad
\begin{tikzpicture}
\node[shape=circle](1) at (4.75,-1.05) {};
\filldraw [black] (1) circle (2pt);
\node[shape=circle](2) at (7.14,-1.05) {};
\filldraw [black] (2) circle (2pt);
\node[shape=circle] (3) at (5.9,-.35) {};
\filldraw [black] (3) circle (2pt);
\node[shape=circle] (4) at (5.9,1.05) {};
\filldraw [black] (4) circle (2pt);
\path[<-] (1) edge (4);
\path[<-] (3) edge (4);
\path[<-] (2) edge (4);
\path[-] (3) edge (2);
\path[<-, bend left=20] (1) edge (2);
\path[<-, bend right=20] (1) edge (2);
\draw[scale=2.2] (1)  to[in=207.5,out=152.5,loop] (1);
\node[] at (5.15,0) {\scriptsize 4};
\node[] at (6.7,0) {\scriptsize 3};
\node[] at (5.75,.3) {\scriptsize 3};
\node[] at (6.5,-.5) {\scriptsize 0};
\node[] at (5.9,-.95) {\scriptsize 1};
\node[] at (5.9,-1.5) {\scriptsize 1};
\node[] at (3.7,-1.05) {\scriptsize 0};
\node[] at (5.9,-1.9) { $\grad f$};
\end{tikzpicture}
\end{center}
\end{example}

\begin{proposition} \label{prop:adj}
The maps $\ndiv: L^2(E,\eps)\to L^2(V,\ups)$ and $\grad: L^2(V,\ups)\to L^2(E,\eps)$ are adjoint.
\end{proposition}
\begin{proof}
For $\mathbf F\in L^2(E,\eps)$ and $g\in L^2(V,\ups)$, we have
\begin{align*}
    \langle \ndiv \mathbf F,g\rangle = \int_V \ndiv\mathbf F g d\ups &= \sum_{v\in V}\sum_{\substack{uv\in E\\ u\neq v}} \frac{\eps(uv)}{\ups(v)} \mathbf F(uv)\cdot\mathbf v_{\vec{uv}} g(v)\ups(v)\\
    &=\sum_{\substack{uv\in E\\ u\neq v}} \mathbf F(uv)\cdot (g(u)\mathbf v_{\vec{vu}} + g(v)\mathbf v_{\vec{uv}})\eps(uv)\\
    &= \sum_{uv\in E} \mathbf F(uv)\cdot (g(v)-g(u))\mathbf v_{\vec{uv}}\eps(uv) = \int_E \mathbf F\cdot\grad gd\eps = \langle \mathbf F,\grad g\rangle,
\end{align*}
so $\ndiv$ and $\grad$ are adjoint.
\end{proof}

Given a linear combination $S=\sum_{i=1}^kc_iv_i$ of vertices in $V$, we define the \textit{characteristic function} of $S$ to be $\ind_S=\sum_{i=1}^kc_i \ind_{v_i}\in L^2(V,\ups)$, where each $\ind_v\in L^2(V,\ups)$ is the scalar field evaluating to $1$ on $v$ and zero on every other vertex. Then, we can define integration on $S$ by
$$\int_S fd\ups = \int_V f\ind_S d\ups$$
for $f\in L^2(V,\ups)$. We also identify $S\subseteq V$ with the linear combination $\sum_{v\in S}v$.

Similarly, given a linear combination $R=\sum_{i=1}^kc_i\vec e_i$ of directed edges in $\vec E$, its \textit{characteristic function} is $\indb_R =\sum_{i=1}^kc_i\indb_{\vec{e}_i}\in L^2(E,\eps)$, where each $\indb_{\vec {uv}}\in L^2(E,\eps)$ is the vector field evaluating to $\mathbf v_{\vec{uv}}$ on $uv$ and zero on every other edge. We can define integration on $R$ by
$$\int_R \mathbf F\cdot d\eps = \int_E \mathbf F\cdot \indb_R d\eps$$
for $\mathbf F\in L^2(E,\eps)$. We also identify $R\subseteq\vec E$ with the linear combination $\sum_{\vec e\in R}\vec e$.

A \textit{walk} $C$ from $u$ to $v$ for vertices $u,v\in V$ is a sequence of directed edges $(\vec e_i)_{i=1}^k$ for which there exists a sequence of vertices $(v_i)_{i=0}^k$ so that $u=v_0$, $v=v_k$, and each $\vec{e_i}$ is a directed edge from $v_{i-1}$ to $v_i$. Note for every vertex $v\in V$ there exists the empty walk from $v$ to itself. Given a walk $C=(\vec e_i)_{i=1}^k$ from $u$ to $v$, the (oriented) \textit{curve} $\frac{C}{\eps}$ from $u$ to $v$ is the linear combination $\sum_{i=1}^k\frac{\vec e_i}{\eps(e_i)}$. A curve $\frac{C}{\eps}$ is \textit{closed} if $v_0=v_k$, and a closed curve is \textit{simple} if the sequence $(v_i)_{i=0}^k$ has repeated vertex $v_0=v_k$ (in particular, $k>0$), and otherwise no other repeated vertices.

Given a curve $\frac{C}{\eps}$ from $u$ to $v$, its \textit{boundary} $\ndiv\frac{C}{\eps}$ is the linear combination $\frac{v}{\ups(v)}-\frac{u}{\ups(u)}$. For $S\subseteq V$, we let $\grad S\subseteq \vec E$ denote the set of all directed edges $\vec {uv}$ satisfying $u\not\in S$ and $v\in S$. The \textit{boundary} of $S$ is the linear combination $-\grad S$. Note these are defined exactly so that for any curve $\frac{C}{\eps}$ and any $S\subseteq V$ that
$$\ndiv\indb_{\frac{C}{\eps}}=\ind_{\ndiv\frac{C}{\eps}}\quad\text{and}\quad\grad\ind_S = \indb_{\grad S}.$$

From these definitions and Proposition~\ref{prop:adj}, we obtain discrete analogues of the divergence and gradient theorems of vector calculus.

\begin{proposition}[Discrete divergence theorem]\label{prop:div}
Let $\mathbf F \in L^2(E,\eps)$ and let $S\subseteq V$. Then
$$\int_{S}-\ndiv \mathbf Fd\ups = \int_{-\grad S} \mathbf F \cdot d\eps=\sum_{\substack{uv\in E\\ u\in S,v\not\in S}} \mathbf F(uv)\cdot \mathbf v_{\vec{uv}} \eps(uv).$$
\end{proposition}

\begin{proposition}[Discrete gradient theorem]\label{prop:grad}
Let $f\in L^2(V,\ups)$ and let $\frac{C}{\eps}$ be a curve from $u$ to $v$. Then
$$\int_{\frac{C}{\eps}}\grad f\cdot d\eps = \int_{\ndiv \frac{C}{\eps}}fd\ups=f(v)-f(u).$$
\end{proposition}

In analogy with the (negative) Laplacian from vector calculus and the (negative) Laplace-Beltrami operator on Riemannian manifolds, we will define the Laplacian on a weighted graph to be the negative divergence of the gradient.

\begin{definition}\label{def:wlap}
The \textit{weighted Laplacian} $L_{(G,\ups,\eps)}:L^2(V,\ups)\to L^2(V,\ups)$ of a weighted multigraph $(G,\ups,\eps)$ is
$$L_{(G,\ups,\eps)}=\ndiv\grad.$$
\end{definition}

Whenever we want to think of $L_{(G,\ups,\eps)}$ as a matrix, we will consider the matrix of $L_{(G,\ups,\eps)}$ in the basis of characteristic functions of vertices. Then the matrix of $L_{(G,\ups,\eps)}$ can be computed to find that
$$L_{(G,\ups,\eps)}=W_{\ups}^{-1}L_{(G,\eps)},$$
where $W_\ups$ is the diagonal matrix with vertex weights as its diagonal entries. Therefore, the combinatorial Laplacian $L_{(G,\eps)}$ can be thought of as the special case where $\ups$ is the counting measure $\#_V$, i.e. all vertex weights are $1$, and the normalised Laplacian $\nL_{(G,\eps)}$ can be thought of as the special case where $\ups$ coincides with the measure obtained by extending the degree map $d_{(G,\eps)}:V\to \mathbb R_{\ge 0}$ to a measure on $V$.

\begin{remark}
The same weighted Laplacian matrix has been studied in \cite{BHG}, as well as a symmetric version in \cite{CWW, XFL}. The Laplace operator $L_{(G,\ups,\eps)}$ coincides with the Laplacian studied by Friedman and Tillich \cite{FT}, who also considered graphs with general vertex and edge measures, and the operator of Horak and Jost \cite{HJ}, who defined a Laplacian for weighted simplicial complexes, for graphs.
\end{remark}

Since $\ndiv$ and $\grad$ are adjoint by Proposition~\ref{prop:adj}, we have for $f,g\in L^2(V,\ups)$ that
$$\langle L_{(G,\ups,\eps)}f,g\rangle = \langle \grad f,\grad g\rangle = \langle f,L_{(G,\ups,\eps)}g\rangle,$$
and so the weighted Laplacian is self-adjoint. Therefore, $L_{(G,\ups,\eps)}$ has all real eigenvalues, and there exists an orthonormal basis for $L^2(V,\ups)$ consisting of eigenfunctions of $L_{(G,\ups,\eps)}$.

The weighted Laplacian Rayleigh quotient of nonzero $f\in L^2(V,\ups)$ is given by
$$\frac{\langle L_{(G,\ups,\eps)}f,f\rangle}{\langle f,f\rangle} = \frac{\langle \grad f,\grad f\rangle}{\langle f,f\rangle} = \frac{\int_E \lVert \grad f\rVert^2d\eps}{\int_V |f|^2d\ups},$$
which is always nonnegative, and is zero if and only if $f$ is constant on every connected component of $G$. Thus all eigenvalues of $L_{(G,\ups,\eps)}$ are nonnegative, and a basis for $\ker L_{(G,\ups,\eps)}$ is given by the characteristic functions of the connected components of $G$.

Because the matrix of the weighted Laplacian is given by $L_{(G,\ups,\eps)}=W_\ups^{-1}(D_{(G,\eps)}-A_{(G,\eps)}),$ we deduce the following.

\begin{lemma}\label{lem:simplify}
Let $(G,\ups,\eps)$ be a weighted multigraph. Then
\begin{itemize}
    \item[(i)] $L_{(G,\ups,\eps)\dju(G',\ups',\eps')} = L_{(G,\ups,\eps)}\oplus L_{(G',\ups',\eps')}$ whenever $(G',\ups',\eps')$ is another weighted graph,
    \item[(ii)]
    $L_{(G-e,\ups,\eps)} = L_{(G,\ups,\eps)}$ whenever $e$ is a loop,
    \item[(iii)]
    $L_{(G-e_1-e_2+e,\ups,\eps')} = L_{(G,\ups,\eps)}$ whenever $e_1,e_2$ and $e$ have the same endpoints, $\eps'$ inherits edge weights from $\eps$ and satisfies $\eps'(e)=\eps(e_1)+\eps(e_2)$,
    \item[(iv)]
    $L_{(G,c\ups,\eps)} = \frac{1}{c}L_{(G,\ups,\eps)}$ whenever $c\in\mathbb R_{>0}$, and
    \item[(v)]
    $L_{(G,\ups,c\eps)} = {c}L_{(G,\ups,\eps)}$ whenever $c\in \mathbb R_{>0}$.
\end{itemize}
\end{lemma}

Another operator we will be interested in is the edge Laplacian, which we define next.

\begin{definition}\label{def:elap}
The \textit{edge Laplacian} $K_{(G,\ups,\eps)}:L^2(E,\eps)\to L^2(E,\eps)$ of a weighted multigraph $(G,\ups,\eps)$ is
$$K_{(G,\ups,\eps)}=\grad\ndiv.$$
\end{definition}

Note that $K_{(G,\ups,\eps)}$ is also self-adjoint. Moreover, because $L_{(G,\ups,\eps)}=\ndiv\grad$ and $K_{(G,\ups,\eps)}=\grad\ndiv$, they share the same nonzero spectrum. The edge Laplacian Rayleigh quotient of $\mathbf F\in L^2(E,\eps)$ is given by
$$\frac{\langle K_{(G,\ups,\eps)}\mathbf F,\mathbf F\rangle}{\langle \mathbf F,\mathbf F\rangle} = \frac{\langle \ndiv\mathbf F,\ndiv\mathbf F\rangle}{\langle \mathbf F,\mathbf F\rangle} = \frac{\int_V |\ndiv\mathbf F|^2d\ups}{\int_E \lVert\mathbf F\rVert^2d\eps}.$$

We next define two important subspaces of $L^2(E,\eps)$. Given a weighted graph $(G,\ups,\eps)$, the \textit{cycle space} is $\cyc(E,\eps) = \ker\ndiv$ and the \textit{cut space} is $\cut(E,\eps) = \im \grad$. See \cite[Section II.3]{Bollobas} for a treatment of these spaces in the unweighted case. Since $\ndiv$ and $\grad$ are adjoint by Proposition~\ref{prop:adj}, the cycle space and the cut space are orthogonal complements.

Moreover, $\ker K_{(G,\ups,\eps)}=\cyc(E,\eps)$, noting that $\ker K_{(G,\ups,\eps)}\supseteq \cyc(E,\eps)$ since $K_{(G,\ups,\eps)}=\grad\ndiv$, and $\ker K_{(G,\ups,\eps)}\subseteq \cyc(E,\eps)$ since $$\langle K_{(G,\ups,\eps)}\mathbf F,\mathbf F\rangle = \langle \ndiv \mathbf F,\ndiv\mathbf F\rangle$$
is zero if and only if $\mathbf F\in \cyc(E,\eps)$. Since $K_{(G,\ups,\eps)}$ is self-adjoint, $\im K_{(G,\ups,\eps)}$ is the orthogonal complement of $\ker K_{(G,\ups,\eps)}=\cyc(E,\eps)$, and so $\im K_{(G,\ups,\eps)} = \cut(E,\eps)$. If $G$ has $n$ vertices, $m$ edges and $k$ components, then the dimension of $\cut(E,\eps)$ is $n-k$, the number of nonzero eigenvalues of $L_{(G,\ups,\eps)}$, and the dimension of $\cyc(E,\eps)$ is $m-n+k$. The former is zero if and only if $G$ has no nonloop edges, and the latter is zero if and only if $G$ is a forest.

The following proposition, which is known in the unweighted case, states that $\cyc(E,\eps)$ and $\cut(E,\eps)$ are spanned by certain vector fields known as \textit{cycles} and \textit{cuts}, respectively, explaining their names. We include a proof here for completeness, which interestingly uses deletion and contraction.

\begin{proposition}\label{prop:cyccut}
Let $(G,\ups,\eps)$ be a weighted graph. Then
\begin{itemize}
    \item[(i)] $\cyc(E,\eps)$ is equal to the subspace of $L^2(E,\eps)$ spanned by vector fields of the form $\indb_{\frac{C}{\eps}}$ for simple closed curves $\frac{C}{\eps}$, and
    \item[(ii)] $\cut(E,\eps)$ is equal to the subspace of $L^2(E,\eps)$ spanned by vector fields of the form $\indb_{\grad S}$ for vertex subsets $S\subseteq V$.
\end{itemize}
\end{proposition}
\begin{proof}
We first note that every vector field of the form $\indb_{\frac{C}{\eps}}$ for some simple closed curve $\frac{C}{\eps}$ lies in $\cyc(E,\eps)$, as 
$$\ndiv \indb_{\frac{C}{\eps}}=\ind_{\ndiv\frac{C}{\eps}}=0.$$
Additionally, every vector field of the form $\indb_{\grad S}$ for some $S\subseteq V$ satisfies
$$\indb_{\grad S}=\grad\ind_{S}\in \cut(E,\eps).$$
It remains to prove that such vector fields span their respective subspaces.

To prove (i), we will proceed by induction on the dimension of $\cyc(E,\eps)$. The base case when $\cyc(E,\eps)$ has dimension zero is immediate.

For the inductive step, if $\cyc(E,\eps)$ has nonzero dimension, then $G$ is not a forest. Then there exists some closed walk $C$ directed along a cycle of $G$. Let $e$ be an edge in $C$. Note $\cyc(E-e,\eps)$ is naturally isomorphic to the subspace of $\cyc(E,\eps)$ evaluating at $e$ to the zero vector. Then a cycle basis for $\cyc(E,\eps)$ can be found by taking a cycle basis of $\cyc(E-e,\eps)$, which exists by the inductive hypothesis, together with $\indb_{\frac{C}{\eps}}$.

To prove (ii), we will similarly proceed by induction on the dimension of $\cut(E,\eps)$. The base case when $\cut(E,\eps)$ has dimension zero is immediate.

For the inductive step, if $\cut(E,\eps)$ has nonzero dimension, then $G$ has some nonloop edge $e$. Note $\cut(E/e,\eps)$, where $E/e$ is the edge set of $G/e$, is naturally isomorphic to the subspace of $\cut(E,\eps)$ evaluating at $e$ to the zero vector. Moreover, cuts in $\cut(E/e,\eps)$ are mapped to under this isomorphism by the cuts in $\cut(E,\eps)$ not separating the endpoints of $e$. Then a cut basis for $\cut(E,\eps)$ can be found by taking a cut basis of $\cut(E/e,\eps)$, which exists by the inductive hypothesis, together with $\indb_{\grad S}$, where $S\subseteq V$ is chosen to contain exactly one of the two endpoints of $e$.
\end{proof}

We conclude this section with a proposition giving some equivalent characterisations of the cut space. We will also call vector fields in $\cut(E,\eps)$ \textit{conservative}, noting that Proposition~\ref{prop:cons} is analogous to a theorem on conservative vector fields in classical vector calculus.

\begin{proposition}\label{prop:cons}
Let $(G,\ups,\eps)$ be a weighted graph and let $\mathbf F\in L^2(E,\eps)$. Then the following are equivalent:
\begin{itemize}
    \item[(i)]
    $\mathbf F\in \cut(E,\eps)$.
    \item[(ii)]
    $\int_{\frac{ C_1}{\eps}}\mathbf  F\cdot d\eps=\int_{\frac{ C_2}{\eps}}\mathbf F\cdot d\eps$ whenever $\frac{C_1}{\eps},\frac{C_2}{\eps}$ are two curves starting at the same vertex $u$ and ending at the same vertex $v$.
    \item[(iii)]
    $\int_{\frac{ C}{\eps}} \mathbf F\cdot d\eps = 0$ for all simple closed curves $\frac{C}{\eps}$.
\end{itemize}
\begin{proof}
We will show that (i) implies (ii), (ii) implies (iii) and (iii) implies (i).

If $\mathbf F\in \cut(E,\eps)$, then $\mathbf F = \grad f$ for some $f\in L^2(V,\ups)$. If $u,v\in V$ and $\frac{C_1}{\eps}$ and $\frac{C_2}{\eps}$ are two curves starting at $u$ and ending at $v$, then by Proposition~\ref{prop:grad},
$$\int_{\frac{C_1}{\eps}}\mathbf F\cdot d\eps = f(v)-f(u) = \int_{\frac{C_2}{\eps}} \mathbf F\cdot d\eps,$$
and so (i) implies (ii).

If (ii) held, then for any simple closed curve $\frac{C}{\eps}$ from $v$ to $v$, we have
$$\int_{\frac{C}{\eps}}\mathbf F\cdot d\eps = \int_{0}\mathbf F\cdot d\eps = 0,$$
where $0$ is used to denote the curve from $v$ to $v$ obtained from the empty walk. So (ii) would imply (iii).

Finally, if
$$\int_{\frac{C}{\eps}}\mathbf F\cdot d\eps = \int_E \mathbf F\cdot \indb_{\frac{C}{\eps}} d\eps = 0$$
for all simple closed curves $\frac{C}{\eps}$, then $\mathbf F$ lies in the orthogonal complement of $\cyc(E,\eps)$, recalling by Proposition~\ref{prop:cyccut} that $\cyc(E,\eps)$ is spanned by the characteristic functions of simple closed curves. Since $\cut(E,\eps)$ is the orthogonal complement of $\cyc(E,\eps)$, it follows that (iii) implies (i).
\end{proof}

\end{proposition}

\section{Deletion-contraction from weighted spanning forests}\label{sec:delcon}
In this section we will prove a deletion-contraction relation for the weighted Laplacian characteristic polynomial. It will be helpful to consider the weighted Laplacian characteristic polynomial scaled by the product of vertex weights.

\begin{definition}
For a weighted multigraph $(G,\ups,\eps)$, define the polynomial
$$\p_{(G,\ups,\eps)}(t)=\det(tW_{\ups}-L_{(G,\eps)})= \left(\prod_{v\in V}\ups(v)\right)\det(tI-L_{(G,\ups,\eps)}).$$
\end{definition}

We will give a combinatorial interpretation for the coefficients of $P_{(G,\ups,\eps)}(t)$ in terms of the spanning forests of $G$. Given a spanning forest $F$ of $G$, define its \textit{weight} to be
$$w(F)=\prod_{T\in F} \left(\sum_{v\in V(T)}\ups(v)\prod_{e\in E(T)}\eps(e)\right),$$
where the product is over all connected components $T$ of $F$. Then the following theorem holds.

\begin{theorem}[Weighted matrix-forest theorem]\label{the:mforest}
Let $(G,\ups,\eps)$ be a weighted multigraph on $n$ vertices with $$\p_{(G,\ups,\eps)}(t)=\sum_{k=1}^{n}(-1)^{n-k}c_kt^k.$$
Then $$c_k=\sum_F w(F),$$
where the sum is over all $k$-component spanning forests $F$.
\end{theorem}
\begin{proof}
It is known that the coefficient of $t^k$ in the characteristic polynomial of $L_{(G,\ups,\eps)}$ multiplied by $(-1)^{n-k}$ is the sum of all $(n-k)\times (n-k)$ principal minors of $L_{(G,\ups,\eps)}$, e.g. by considering the Laplace expansion of $\det(tI-L_{(G,\ups,\eps)})$.

For $S\subseteq V$ of cardinality $k$, the $(n-k)\times (n-k)$ principal submatrix of $L_{(G,\ups,\eps)}$ obtained by deleting the rows and columns indexed by $S$ is equal to the corresponding principal submatrix of $L_{(G,\eps)}$ with each row corresponding to $v\in S^c$ divided by $\ups(v)$. It follows then that the associated $(n-k)\times(n-k)$ principal minor of $L_{(G,\ups,\eps)}$ is equal to the corresponding principal minor of $L_{(G,\eps)}$ divided by $\prod_{v\in S^c}\ups(v)$.

Therefore,
$$c_k = \left(\prod_{v\in V}\ups(v)\right)\sum_{\substack{S\subseteq V \\ |S|=k}} \frac{\det L_{(G,\eps)}^{\hat S}}{\prod_{v\in S^c}\ups(v)}=\sum_{\substack{S\subseteq V\\|S|=k}}\left(\prod_{v\in S}\ups(v)\right)\det L_{(G,\eps)}^{\hat S}.$$
By Theorem~\ref{the:mtree}, the principal minors matrix-tree theorem, this is equal to the sum over all $k$-rooted spanning forests of $G$ of the product of root weights and edge weights.

For a $k$-component spanning forest $F$ of $G$, a choice of roots is made by choosing a single vertex $v$ of each component $T$ of $G$. The sum over all choices of roots for $F$ of the product of root weights is equal to $\prod_{T\in F}\sum_{v\in V(T)}\ups(v)$.

Putting this all together, we find that
$$c_k = \sum_{F} \left(\left(\prod_{T\in F} \sum_{v\in V(T)}\ups(v)\right)\prod_{e\in E(F)}\eps(e)\right) = \sum_F w(F),$$
where the sum is over all $k$-component spanning forests $F$ of $G$, as desired.
\end{proof}

Various polynomials and invariants associated with a graph satisfy a recurrence involving the \textit{deletion} and \textit{contraction} of an edge. For example, the number of spanning trees $\tau(G)$ of a graph $G$ satisfies for every nonloop edge $e$
$$\tau(G)=\tau(G-e) +\tau(G/e),$$
where $G-e$ is the graph obtained from $G$ by removing the edge $e$, and $G/e$ is the graph obtained from $G$ by removing $e$ and formally identifying its endpoints as a single vertex in $G/e$. The deletion-contraction formula for $\tau(G)$ follows from the existence of a bijection between the spanning trees of $G$ not containing $e$ and the spanning trees of $G-e$, together with a bijection between the spanning trees of $G$ containing $e$ and the spanning trees of $G/e$.

The combinatorial Laplacian and normalised Laplacian characteristic polynomials are not known to satisfy a deletion-contraction recurrence. The introduction of vertex weights, however, gives a deletion-contraction recurrence for the weighted Laplacian characteristic polynomial. If $\ups$ is a weight function on the vertices of $G$, then $\ups/e$ is the weight function on the vertices of $G/e$ where the weight of the new vertex obtained by contracting $e$ is equal to the sum of the weights of its constituents under $\ups$, and the weights of the other vertices are given by their original weights under $\ups$.

\begin{example}
Let $(G,\ups,\eps)$ the the weighted graph depicted on the left, and let $e$ be the edge of weight $3$ in $(G,\ups,\eps)$. Then $(G/e,\ups/e,\eps)$, obtained from $(G,\ups,\eps)$ by contracting the edge $e$, is the weighted graph depicted on the right. Note that the endpoints of $e$, which had weights $2$ and $1$ in $(G,\ups,\eps)$, are replaced by a single vertex of weight $2+1=3$ in $(G/e,\ups/e,\eps)$.
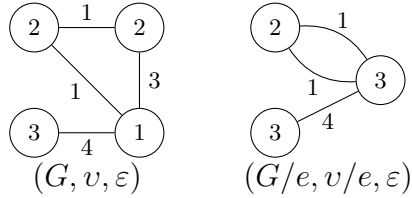
\begin{center}
\begin{tikzpicture}
\node[shape=circle, draw=black](1) at (0,0) {\scriptsize 3};
\node[shape=circle, draw=black](2) at (0,1.4) {\scriptsize 2};
\node[shape=circle,  draw=black] (3) at (1.4,0) {\scriptsize 1};
\node[shape=circle,  draw=black] (4) at (1.4,1.4) {\scriptsize 2};
\path[-] (1) edge (3);
\path[-] (3) edge (4);
\path[-] (2) edge (4);
\path[-] (3) edge (2);
\node[] at (.7,1.6) {\scriptsize $1$};
\node[] at (1.6,.7) {\scriptsize$3$};
\node[] at (.55,.55) {\scriptsize$1$};
\node[] at (.7,-.2) {\scriptsize$4$};
\node[] at (.7,-.6) {$(G,\ups,\eps)$};
\end{tikzpicture}\quad\quad
\begin{tikzpicture}
\node[shape=circle, draw=black](1) at (0,0) {\scriptsize 3};
\node[shape=circle, draw=black](2) at (0,1.4) {\scriptsize 2};
\node[shape=circle,  draw=black] (3) at (1.4,.7) {\scriptsize 3};
\path[-] (1) edge (3);
\path[-, bend left] (2) edge (3);
\path[-, bend left] (3) edge (2);
\node[] at (.9,1.5) {\scriptsize $1$};
\node[] at (.5,.6) {\scriptsize$1$};
\node[] at (.7,.15) {\scriptsize$4$};
\node[] at (.7,-.6) {$(G/e,\ups/e,\eps)$};
\end{tikzpicture}
\end{center}
\end{example}

With these definitions in mind, we will show the following.

\begin{theorem}\label{the:delcon}
Let $(G,\ups,\eps)$ be a weighted multigraph and let $e$ be a nonloop edge. Then
$$\p_{(G,\ups,\eps)}(t)=\p_{(G-e,\ups,\eps)}(t)-\eps(e)\p_{(G/e,\ups/e,\eps)}(t).$$
\end{theorem}
\begin{proof}
We first note that there exists a bijection between the $k$-component spanning forests of $G$ not containing $e$ and the $k$-component spanning forests of $G-e$, by sending a spanning forest $F$ to its image in $G-e$. Since vertex weights and edge weights are preserved, this bijection preserves the weight of each spanning forest.

Additionally, since $e$ is not a loop, there is a bijection between the $k$-component spanning forests of $G$ containing $e$ and the $k$-component spanning forests of $G/e$, by sending a spanning forest $F$ to its contraction $F/e$. The vertex and edge weights of each component of $F$ not containing $e$ are preserved.

As for the component $T$ of $F$ containing $e$, its image $T/e$ under contraction has vertex weights preserved, except the image of the endpoints of $e$ are replaced with a single vertex whose weight is the sum of the weights of the pair of endpoints. All edges of $T$ except $e$ appear in $T/e$ with edge weights preserved. Therefore, we have
$$\sum_{v\in V(T)}\ups(v)\prod_{e'\in E(T)}\eps(e') = \eps(e) \left(\sum_{v\in V(T/e)}\ups(v) \prod_{e'\in E(T/e)}\eps(e')\right).$$
It follows then that for all $k$-component spanning forests $F$ of $G$ containing $e$,
$$w(F)=\eps(e)w(F/e).$$

Let $G$ have $n$ vertices, and let $(-1)^{n-k}c_k^{(G,\ups,\eps)}$, $(-1)^{n-k}c_k^{(G-e,\ups,\eps)}$ and $(-1)^{n-k-1}c_k^{(G/e,\ups/e,\eps)}$ denote the coefficients of $t^k$ in $\p_{(G,\ups,\eps)}(t)$, $\p_{(G-e,\ups,\eps)}(t)$ and $\p_{(G/e,\ups/e,\eps)}(t)$, respectively. From the bijections constructed together with Theorem~\ref{the:mforest}, it follows that
$$c_k^{(G,\ups,\eps)}=c_k^{(G-e,\ups,\eps)}+\eps(e)c_k^{(G/e,\ups/e,\eps)}.$$
Since this holds for all $k$, we conclude that
$$\p_{(G,\ups,\eps)}(t) = \p_{(G-e,\ups,\eps)}(t) - \eps(e)\p_{(G/e,\ups/e,\eps)}(t).$$
\end{proof}

\section{Interlacing theorems}\label{sec:interlace}
In this section, we will prove interlacing theorems relating to deletion and contraction for weighted graphs, and then apply these to obtain interlacing results for combinatorial and normalised Laplacian eigenvalues.

We begin by proving a lemma relating the weighted Laplacian eigenvalues of a weighted graph $(G,\ups,\eps)$ and those of $(G-e,\ups,\eps)$. We follow the proof of Grone, Merris and Sunder in \cite[Theorem 4.1]{GMS}, which is the special case when the graph is unweighted.

\begin{lemma}\label{lem:G-e}
Let $(G,\ups,\eps)$ be a weighted multigraph and let $e$ be an edge. Let $\lambda_1\le\dots\le\lambda_n$ and $\mu_1\le\dots\le\mu_n$ denote the weighted Laplacian eigenvalues of $(G,\ups,\eps)$ and $(G-e,\ups,\eps)$, respectively. Then
$$\mu_1\le\lambda_1\le\dots\le\mu_n\le\lambda_n.$$
\end{lemma}
\begin{proof}
It suffices to prove the inequalities for the nonzero eigenvalues, so we will prove the interlacing for the eigenvalues of the edge Laplacian.

Let $\iota: L^2(E-e,\eps)\to L^2(E,\eps)$ be the isometric natural inclusion, where $E-e$ denotes the edge set of $G-e$. Then 
$$K_{(G-e,\ups,\eps)}=\iota^*K_{(G,\ups,\eps)}\iota,$$
since the matrix of $K_{(G,\ups,\eps)}$ with respect to the orthogonal basis of vector fields of the form $\indb_{\vec{uv}}$ (after arbitrarily orienting each edge) has the matrix of $K_{(G-e,\ups,\eps)}$ as a principal submatrix.

By Theorem~\ref{the:cauchy}, the eigenvalues of $K_{(G,\ups,\eps)}$ interlace the eigenvalues of $K_{(G-e,\ups,\eps)}$. Since $K_{(G,\ups,\eps)}$ has one more eigenvalue than $K_{(G-e,\ups,\eps)}$, we conclude that
$$\mu_1\le\lambda_1\le\dots\le\mu_n\le\lambda_n.$$
\end{proof}
\begin{remark}
By Lemma~\ref{lem:simplify}, the interlacing result in Lemma~\ref{lem:G-e} also applies whenever we reduce the weight of an edge, even though its proof required that an edge be deleted.
\end{remark}

We will prove many interlacing results in this section. However, beyond the proof of Lemma~\ref{lem:G-e}, we will not require any applications of Lemma~\ref{lem:rq} or Theorem~\ref{the:cauchy}. We will also not require the use of Rayleigh quotients or any computations.

The following lemma on interlacing roots of real-rooted polynomials will help us prove an interlacing theorem relating the eigenvalues of a weighted graph and the eigenvalues of the weighted graphs obtained by deleting and contracting an edge.

\begin{lemma} \label{lem:ivterlace}
Let
$$\mu_1\le\lambda_1\le\dots\le \mu_n\le\lambda_n$$ be real numbers and suppose
$$f(t)=\prod_{i=1}^n(t-\lambda_i)-\prod_{i=1}^n(t-\mu_i)$$ is nonzero. Then $f(t)$ has $n-1$ real roots $\nu_1\le\dots\le\nu_{n-1}$ satisfying
$$\mu_1\le\lambda_1\le\nu_1\le\cdots\le\nu_{n-1}\le\mu_n\le\lambda_n.$$
\end{lemma}
\begin{proof}
We will proceed by induction.

For the base case, suppose
$$\mu_1<\lambda_1<\dots<\mu_n<\lambda_n.$$
Note $$ \sgn (f(\lambda_j)) = -\prod_{i=1}^n\sgn(\lambda_j-\mu_i)=(-1)^{n-j+1}$$ and similarly $$\sgn (f(\mu_j)) = \prod_{i=1}^n\sgn(\mu_j-\lambda_i)=(-1)^{n-j+1}.$$
By the intermediate value theorem, $f(t)$ has a real root in each of the open intervals $(\lambda_j,\mu_{j+1})$. Since nonzero $f(t)$ has degree at most $n-1$ and since there are $n-1$ disjoint intervals of the form $(\lambda_j,\mu_{j+1})$, it follows that $f(t)$ has $n-1$ real roots $\nu_1\le\dots\le\nu_{n-1}$ satisfying
$$\mu_1<\lambda_1<\nu_1<\dots<\nu_{n-1}<\mu_n<\lambda_n.$$

For the inductive step, suppose there is some $\mu_j=\lambda_j$ or $\lambda_j=\mu_{j+1}$. Let $\xi$ be the shared value, and denote the remaining numbers by
$$\mu_1'\le\lambda_1'\le\dots\le\mu_{n-1}'\le\lambda_{n-1}'.$$
Since $f(t)$ is a nonzero polynomial with $\xi$ as a root, $\frac{f(t)}{t-\xi}$ is a nonzero polynomial. Applying the inductive hypothesis to
$$\frac{f(t)}{t-\xi} = \prod_{i=1}^{n-1}(t-\lambda_i')-\prod_{i=1}^{n-1}(t-\mu_i'),$$
it follows that $\frac{f(t)}{t-\xi}$ has $n-2$ roots $\nu_1'\le\dots\le\nu_{n-2}'$ satisfying
$$\mu_1'\le\lambda_1'\le\nu_1'\le\dots\le\nu_{n-2}'\le\mu_{n-1}'\le\lambda_{n-1}'.$$
Inserting $\xi\le\xi\le\xi$ into the above chain of inequalities completes the proof of the inductive step.
\end{proof}

We are now ready to prove the following theorem.

\begin{theorem}\label{the:conlace}
Let $(G,\ups,\eps)$ be a weighted multigraph and let $e$ be a nonloop edge. Let $\lambda_1\le\dots\le\lambda_n$, $\mu_1\le\dots\le\mu_n$, and $\nu_1\le\dots\le\nu_{n-1}$ denote the weighted Laplacian eigenvalues of $(G,\ups,\eps)$, $(G-e,\ups,\eps)$ and $(G/e,\ups/e,\eps)$, respectively. Then
$$\mu_1\le\lambda_1\le\nu_1\le\dots\le\nu_{n-1}\le\mu_n\le\lambda_n.$$
\end{theorem}
\begin{proof}
By Theorem~\ref{the:delcon},
$$\p_{(G,\ups,\eps)}(t)=\p_{(G-e,\ups,\eps)}(t)-\eps(e)\p_{(G/e,\ups/e,\eps)}(t).$$
Note that the Laplacian eigenvalues $\nu_1\le\dots\le\nu_{n-1}$ of $(G/e,\ups/e,\eps)$ are exactly the roots of the nonzero polynomial
$$-\frac{\eps(e) \p_{(G/e,\ups/e,\eps)}(t)}{\prod_{v\in V}\ups(v)} = \prod_{i=1}^n(t-\lambda_i)-\prod_{i=1}^n(t-\mu_i).$$

By Lemma~\ref{lem:G-e}, the Laplacian eigenvalues of $(G,\ups,\eps)$ and $(G-e,\ups,\eps)$ satisfy
$$\mu_1\le\lambda_1\le\dots\le\mu_n\le\lambda_n.$$
Applying Lemma~\ref{lem:ivterlace}, we conclude that
$$\mu_1\le\lambda_1\le\nu_1\le\dots\le\nu_{n-1}\le\mu_n\le\lambda_n.$$
\end{proof}

Note $(G/e,\ups/e,\eps) = (G-e,\ups,\eps)/\sim$, where $\sim$ is the equivalence relation with the endpoints of $e$ in the same equivalence class and no other relations. Thus the interlacing between the weighted Laplacian eigenvalues of $(G-e,\ups,\eps)$ and $(G/e,\ups/e,\eps)$ in Theorem~\ref{the:conlace} really describes an interlacing result between the weighted Laplacian eigenvalues of a weighted graph and its quotient weighted graph under an equivalence relation with only two distinct vertices related under $\sim$. Repeated applications of this interlacing theorem then immediately gives us the following.

\begin{corollary}\label{cor:quotient}
Let $(G,\ups,\eps)$ be a weighted multigraph and let $\sim$ be an equivalence relation on its vertices. Let $\lambda_1\le\dots\le\lambda_n$ and $\mu_1\le\dots\le\dots\le\mu_{n-k}$ denote the weighted Laplacian eigenvalues of $(G,\ups,\eps)$ and $(G,\ups,\eps)/\sim$, respectively. Then
$$\lambda_i\le\mu_i\le\lambda_{i+k}$$
for all $1\le i \le n-k$.
\end{corollary}
\begin{remark}
Corollary~\ref{cor:quotient} could have alternatively been proved by an application of Theorem~\ref{the:cauchy}, the Cauchy interlacing theorem, with the isometry $\iota:L^2(V/\sim, \ups)\to L^2(V,\ups)$ sending the characteristic function of a vertex in $V/\sim$ to the characteristic function of the corresponding equivalence class of vertices in $V$.
\end{remark}

\begin{corollary}\label{cor:vweight}
Let $(G,\ups,\eps)$ be a weighted multigraph and let $\ups':V\to \mathbb R_{>0}$ be obtained from $\ups$ by decreasing the weight of a single vertex. Let $\lambda_1\le\dots\le\lambda_n$ and $\mu_1\le\dots\le\mu_n$ denote the weighted Laplacian eigenvalues of $(G,\ups,\eps)$ and $(G,\ups',\eps)$, respectively. Then
$$\lambda_1\le\mu_1\le\dots\le\lambda_n\le\mu_n.$$
\end{corollary}
\begin{proof}
Let $v\in V$ be the vertex satisfying $\ups'(v)<\ups(v)$.

Consider the weighted graph with a single vertex of weight $\ups(v)-\ups'(v)$ and no edges, which has a zero as its sole weighted Laplacian eigenvalue. The disjoint union of this weighted graph and $(G,\ups',\eps)$ has $(G,\ups,\eps)$ as a quotient weighted graph, under the equivalence relation $\sim$ relating the isolated vertex and $v$.

By Corollary~\ref{cor:quotient}, the weighted Laplacian eigenvalues of the disjoint union interlace those of $(G,\ups,\eps)$. By Lemma~\ref{lem:simplify}(i), the weighted Laplacian eigenvalues of the disjoint union are exactly the weighted Laplacian eigenvalues of $(G,\ups',\eps)$ together with an additional zero. We conclude that
$$\lambda_1\le\mu_1\le\dots\le\lambda_n\le\mu_n.$$
\end{proof}

To state some of the results in this section, we will let the statement
$$\lambda_{i-a}\le \mu_i \le \lambda_{i+b}$$
be shorthand to mean that both $\lambda_{i-a}\le \mu_i$ if $\lambda_{i-a}$ is defined and $\mu_i\le \lambda_{i+b}$ if $\lambda_{i+b}$ is defined.

Grone, Merris and Sunder also studied the effect of edge contraction on Laplacian eigenvalues. Specifically, they gave an interlacing result for the combinatorial Laplacian eigenvalues of unweighted simple graphs before and after edge contraction. The following result recovers their \cite[Theorem 4.9]{GMS} as a special case.

\begin{corollary}\label{cor:cquot}
Let $(G,\eps)$ be an edge-weighted multigraph and let $\sim$ be an equivalence relation on its vertices. Let $\lambda_1\le\dots\le\lambda_n$ and $\mu_1\le\dots\le\mu_{n-k}$ denote the combinatorial Laplacian eigenvalues of $(G,\eps)$ and $(G/\sim,\eps)$, respectively. Then
$$\lambda_{i}\le \mu_i \le \lambda_{i+q+k}$$
for all $1\le i \le n-k$, where $q$ is the number of equivalence classes of $\sim$ with more than one element.
\end{corollary}
\begin{proof}
Let $\mu_1'\le\dots\le \mu_{n-k}'$ denote the weighted Laplacian eigenvalues of $(G,\#_V,\eps)/\sim$. By Corollary~\ref{cor:quotient}, we have
$$\lambda_i\le \mu_i'\le \lambda_{i+k}$$
for all $1\le i \le n-k$. After applying Corollary~\ref{cor:vweight} multiple times in succession, once for each of the $q$ vertices $G/\sim$ with weight $>1$ under $\#_V/\sim$, it also follows that
$$\mu_{i+q}\le \mu_{i+q}'\le \lambda_{i+q+k}$$
for all $1\le i \le n-q-k$, as desired.
\end{proof}

We can also deduce from our work Butler's \cite[Theorem 3.2]{Butler}, which improved a result of Chen, Davis, Hall, Li, Patel and Stewart in \cite[Theorem 2.7]{CDHLPS} relating normalised Laplacian eigenvalues before and after identifying two vertices of a graph.

\begin{corollary}\cite[Theorem 3.2]{Butler} \label{cor:nquot}
Let $(G,\eps)$ be an edge-weighted multigraph with no isolated vertices and let $\sim$ be an equivalence relation on its vertices. Let $\lambda_1\le\dots\le\lambda_n$ and $\mu_1\le\dots\le\mu_{n-k}$ denote the normalised Laplacian eigenvalues of $(G,\eps)$ and $(G/\sim,\eps)$, respectively. Then
$$\lambda_{i}\le \mu_i \le \lambda_{i+k}$$
for all $1\le i \le n-k$.
\end{corollary}
\begin{proof}
Recall that the normalised Laplacian eigenvalues of $(G,\eps)$ are exactly the weighted Laplacian eigenvalues of $(G,d_{(G,\eps)},\eps)$ and, similarly, the normalised Laplacian eigenvalues of $(G/\sim,\eps)$ are the weighted Laplacian eigenvalues of $(G/\sim,d_{(G/\sim,\eps)},\eps)$.

The result then follows by applying Corollary~\ref{cor:quotient} to the weighted graph $(G,d_{(G,\eps)},\eps)$ and the equivalence relation $\sim$, while noting that the vertex measure $d_{(G,\eps)}/\sim$ coincides with the measure $d_{(G/\sim,\eps)}$.
\end{proof}

Given two graphs $G=(V,E)$ and $G'=(V',E')$, we define $G+G'$ to be the graph $(V\cup V', E\cup E')$. Note we do not require $V$ and $V'$ are disjoint, nor $E$ and $E'$. The \textit{merge} of weighted graphs $(G,\ups,\eps)$ and $(G',\ups',\eps')$ is the weighted graph $(G+G',\ups+\ups',\eps+\eps')$, where if a weight function is not defined on a vertex or an edge, we take it to be zero in the sum.

Our next theorem describes an interlacing relationship between the weighted Laplacian eigenvalues of $(G,\ups,\eps)$ and $(G+G',\ups+\ups',\eps+\eps')$. We will assume without loss of generality that $E$ and $E'$ are disjoint, as doing so does not change the weighted Laplacian eigenvalues of $(G+G',\ups+\ups',\eps+\eps')$ by Lemma~\ref{lem:simplify}(iii).

\begin{theorem}\label{the:merge}
Let $(G,\ups,\eps)$ and $(G',\ups',\eps')$ be weighted multigraphs. Let $\lambda_1\le\dots\le \lambda_n$ and $\mu_1\le\dots\le \mu_{n+k}$ denote the weighted Laplacian eigenvalues of $(G,\ups,\eps)$ and $(G+G',\ups+\ups',\eps+\eps')$, respectively. Then
$$\mu_{i-s+c+k}\le\lambda_i\le \mu_{i+s-b}$$
for all $1\le i \le n$, where $G'$ has $s$ vertices and $c$ components, and $b$ weighted Laplacian eigenvalues of $(G',\ups',\eps')$ are $\ge \lambda_n$.
\end{theorem}
\begin{proof}
Note $(G+G',\ups+\ups',\eps+\eps')$ is the quotient weighted graph of $(G,\ups,\eps)\dju (G',\ups',\eps')$ over $\sim$, where $\sim$ identifies the shared vertices in $V$ and $V'$. By Lemma~\ref{lem:simplify}(i), the weighted Laplacian eigenvalues of $(G,\ups,\eps)\dju (G',\ups',\eps')$ are exactly the disjoint union of those of $(G,\ups,\eps)$ and $(G',\ups',\eps')$.

By Corollary~\ref{cor:quotient}, each $\mu_i$ is greater than or equal to the $i$th smallest eigenvalue among the weighted Laplacian eigenvalues of $(G,\ups,\eps)$ and $(G',\ups',\eps')$. When $1\le i \le n+s-b$, this is equal to the $i$th smallest eigenvalue among the weighted Laplacian eigenvalues of $(G,\ups,\eps)$ and the $s-b$ of those of $(G',\ups',\eps')$, after excluding the $b$ weighted Laplacian eigenvalues that are $\ge \lambda_n$. A lower bound for this when $1+s-b\le i \le n+s-b$ is $\lambda_{i-s+b}$. Therefore, we have
$$\lambda_i\le \mu_{i+s-b}$$
for all $1\le i \le \min\{n,n-s+b+k\}$.

Similarly, by Corollary~\ref{cor:quotient}, each $\mu_{n-i+k}$ is less than or equal to the $i$th largest eigenvalue among the weighted Laplacian eigenvalues of $(G,\ups,\eps)$ and $(G',\ups',\eps')$. Since $(G',\ups',\eps')$ has $c$ zero eigenvalues, this is equal to the $i$th largest eigenvalue among the weighted Laplacian eigenvalues of $(G,\ups,\eps)$ and the $s-c$ nonzero weighted Laplacian eigenvalues of $(G',\ups',\eps')$ when $1\le i \le n+s-c$. This is upper bounded by $\lambda_{n-i+s-c}$ when $1+s-c\le i\le n+s-c$. Therefore,
$$\mu_{i-s+c+k}\le \lambda_i$$
for all $\max\{1,1+s-c-k\}\le i\le n$, completing the proof of the theorem.
\end{proof}

We next prove an interlacing result on weighted Laplacian eigenvalues relating to subgraph deletion.

\begin{corollary}\label{cor:subgraph}
Let $(G,\ups,\eps)$ be a weighted multigraph. Let $R\subseteq E$ and let $S\subseteq V$ be a subset of the isolated vertices of $G-R$. Let $\lambda_1\le\dots\le\lambda_n$ and $\mu_1\le\dots\le\mu_{n-k}$ denote the weighted Laplacian eigenvalues of $(G,\ups,\eps)$ and $(G-R-S,\ups,\eps)$, respectively. Then
$$\lambda_{i-s+c+k}\le\mu_{i}\le\lambda_{i+k}$$
for all $1\le i\le n-k$, where $G[R]$ has $s$ vertices and $c$ components.
\end{corollary}
\begin{proof}
We will assume without loss of generality that $S=\emptyset$. This is sufficient to prove the general case, since the weighted Laplacian eigenvalues of $(G-R,\ups,\eps)$ are exactly those of $(G-R-S,\ups,\eps)$ together with $k$ additional zero eigenvalues.

By applying Lemma~\ref{lem:G-e} to each edge in $R$, we find that 
$$\mu_i\le \lambda_i$$
for all $1\le i \le n$. This proves one direction of the interlacing inequalities. The other direction requires a little more work.

For $0<\delta < \min_{v\in V(G[R])}\ups(v)$, we will define the measure $\ups_\delta$ on the vertices of $G[R]$ to send each vertex to $\delta$. Then, $(G,\ups,\eps)$ can be seen as the merge of $(G-R,\ups-\ups_\delta,\eps)$ and $(G[R],\ups_\delta,\eps)$. Let $\mu_1^{\delta}\le\dots \le \mu_n^{\delta}$ denote the weighted Laplacian eigenvalues of $(G-R,\ups-\ups_\delta,\eps)$. By Theorem~\ref{the:merge}, each
$$\lambda_{i-s+c}\le \mu_i^\delta$$
for all $1+s-c\le i\le n$. Since
$$\lim_{\delta\to 0} L_{(G-R,\ups-\ups_\delta,\eps)}=L_{(G-R,\ups,\eps)}$$
as matrices, it also follows that
$$\lambda_{i-s+c}\le\lim_{\delta\to 0} \mu_i^\delta =\mu_i$$
for all $1+s-c\le i \le n$, completing the proof of the second direction.
\end{proof}
\begin{remark}
Because vertex weights are unmodified, the statement of Corollary~\ref{cor:subgraph} also applies directly to the combinatorial Laplacian eigenvalues of edge-weighted multigraphs.
\end{remark}

Butler proved an interlacing theorem on the normalised Laplacian eigenvalues of an edge-weighted graph before and after deleting a subgraph in \cite[Theorem 1.2]{Butler}, which generalised a version for edge deletion by Chen, Davis, Hall, Li, Patel and Stewart in \cite[Theorem 2.3]{CDHLPS}. Our techniques recover Butler's result and improves it whenever one or more vertices are removed in deleting the subgraph.

\begin{corollary}\label{cor:nsubgraph}
Let $(G,\eps)$ be an edge-weighted multigraph with no isolated vertices. Let $R\subseteq E$ and let $S\subseteq V$ consist of all isolated vertices of $G-R$. Let $\lambda_1\le\dots\le\lambda_n$ and $\mu_1\le\dots\le \mu_{n-k}$ denote the normalised Laplacian eigenvalues of $(G,\eps)$ and $(G-R-S,\eps)$, respectively. Then
$$\lambda_{i-s+c+k}\le\mu_{i}\le\lambda_{i+s-b}$$
for all $1\le i \le n-k$, where $G[R]$ has $s$ vertices and $c$ components, $b$ of which are bipartite.
\end{corollary}
\begin{proof}
Note that $(G,d_{(G,\eps)},\eps)$ is the merge of $(G-R-S,d_{(G-R-S,\eps)},\eps)$ and $(G[R],d_{(G[R],\eps)},\eps)$. It is also known that the normalised Laplacian eigenvalues of an edge-weighted multigraph all lie in the interval $[0,2]$, and the multiplicity of $2$ is equal to the number of bipartite components. Therefore, at least $b$ normalised Laplacian eigenvalues of $(G[R],\eps)$ are $\ge \mu_{n-k}$. By Theorem~\ref{the:merge},
$$\lambda_{i-s+c+k}\le \mu_i\le \lambda_{i+s-b}$$
for all $1\le i \le n-k$.
\end{proof}

\section{Addition-reduction from equivalent electrical networks}\label{sec:elec}
In Section~\ref{sec:calc}, we defined conservative vector fields on weighted graphs. Classical conservative vector fields often arise naturally from physical forces satisfying conservation of energy. One instance of conservative vector fields on graphs can be found in Brooks, Smith, Stone and Tutte's work on square tilings of rectangles, in which they show such tilings can be associated to electrical circuits \cite[Section 1]{BSST}. Each electrical circuit may be realised as a conservative vector field on a graph. We will give a version of their result for rectangular tilings of rectangles.

All rectangles in this paper are assumed to exist in the plane and have positive dimensions, with two sides parallel to the horizontal axis, and two sides parallel to the vertical axis. Hence, the \textit{height} $h(r)$ and \textit{width} $w(r)$ of a rectangle $r$ are well-defined.

A \textit{rectangular tiling} of a rectangle $R$ is a subdivision of $R$ into finitely many interior-disjoint rectangles $r$. Given a rectangular tiling, we can associate an edge-weighted graph $(G,\eps)$, constructed as follows. The vertex set of $G$ consists of all horizontal lines in the plane which cover some horizontal side of a rectangle in the tiling, and the edge set of $G$ consists of the rectangles used in the subdivision. Each edge connects the vertices corresponding to the top and bottom sides of the rectangle. The edge weight of a rectangle $r$ is given by
$$\eps(r) = \frac{w(r)}{h(r)}.$$

Moreover, there is a natural vector field $\mathbf F$ on $(G,\eps)$ associated with the tiling, where for each rectangle $r=uv$ in the subdivision with bottom side corresponding to vertex $u$ and top side corresponding to vertex $v$, we have
$$\mathbf F(r) = h(r) \mathbf v_{\vec uv}.$$

\begin{example}
Consider the rectangular tiling of the rectangle $R$ of height $3$ and width $6$, embedded in the plane, as depicted on the left. It is associated with the edge-weighted graph $(G,\eps)$ depicted in the centre and the vector field $\mathbf F\in L^2(E,\eps)$ depicted on the right.
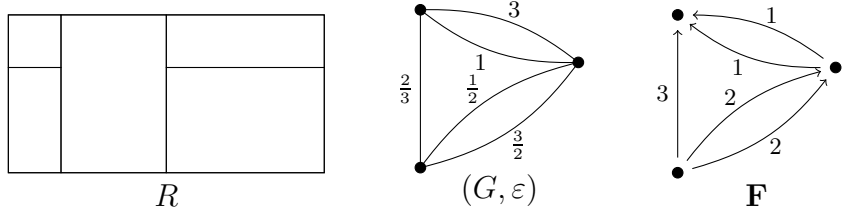
\begin{center}
\begin{tikzpicture}
\draw[draw=black] (0,0) rectangle ++(4.2,2.1);
\draw[draw=black] (0,0) rectangle ++(.7,1.4);
\draw[draw=black] (0,1.4) rectangle ++(.7,.7);
\draw[draw=black] (.7,0) rectangle ++(1.4,2.1);
\draw[draw=black] (2.1,0) rectangle ++(2.1,1.4);
\draw[draw=black] (2.1,1.4) rectangle ++(2.1,.7);
\node[] at (2.1,-.3) {$R$};
\end{tikzpicture}\quad\quad
\begin{tikzpicture}
\coordinate (1) at (0,0);
\coordinate (2) at (0,2.1);
\coordinate (3) at (2.1,1.4);
\filldraw [black] (1) circle (2pt);
\filldraw [black] (2) circle (2pt);
\filldraw [black] (3) circle (2pt);
\draw[-] (1) to (2);
\draw[-,bend left=20] (3) to (2);
\draw[-,bend right=20] (3) to (2);
\draw[-,bend left=20] (1) to (3);
\draw[-, bend right=20] (1) to (3);
\node[] at (-.2,1.05) {\scriptsize $\frac{2}{3}$};
\node[] at (.8,1.4) {\scriptsize $1$};
\node[] at (.7,1.05) {\scriptsize $\frac{1}{2}$};
\node[] at (1.3,.3) {\scriptsize $\frac{3}{2}$};
\node[] at (1.25,2.1) {\scriptsize $3$};
\node[] at (1.05,-.3) {$(G,\eps)$};
\end{tikzpicture}\quad\quad 
\begin{tikzpicture}
\node[shape=circle](1) at (0,0) {};
\node[shape=circle](2) at (0,2.1) {};
\node[shape=circle](3) at (2.1,1.4) {};
\filldraw [black] (1) circle (2pt);
\filldraw [black] (2) circle (2pt);
\filldraw [black] (3) circle (2pt);
\draw[->] (1) to (2);
\draw[->,bend left=18.75] (3) to (2);
\draw[->,bend right=18.75] (3) to (2);
\draw[->,bend left=18.75] (1) to (3);
\draw[->, bend right=18.75] (1) to (3);
\node[] at (-.2,1.05) {\scriptsize $3$};
\node[] at (.8,1.4) {\scriptsize $1$};
\node[] at (.7,1.0) {\scriptsize $2$};
\node[] at (1.3,.35) {\scriptsize $2$};
\node[] at (1.25,2.1) {\scriptsize $1$};
\node[] at (1.05,-.3) {$\mathbf F$};
\end{tikzpicture}
\end{center}
\end{example}

\begin{lemma}\label{lem:kirchhoff}
Let $(G,\eps)$ and $\mathbf F$ be the edge-weighted multigraph and vector field, respectively, associated with a rectangular tiling of a rectangle $R$. Let $a$ and $b$ denote the vertices associated with the bottom and top sides of $R$, respectively, and let $\ups:V\to\mathbb R_{>0}$ be arbitrary. Then $\mathbf F$ satisfies \textit{Kirchhoff's laws} with poles $a,b$. That is,
\begin{enumerate}
    \item[(i)] $\ndiv \mathbf F(v) = 0$ for all $v\in \{a,b\}^c$, and
    \item[(ii)] $\int_{\frac{C}{\eps}} \mathbf F\cdot d\eps = 0$ for all simple closed curves $\frac{C}{\eps}$.
\end{enumerate}
\end{lemma}
\begin{proof}
For $v\in V$, the contribution of a rectangle $r$ with top side on $v$ to $\ndiv \mathbf F(v)$ is
$$\frac{1}{\ups(v)}h(r)\frac{w(r)}{h(r)}=\frac{1}{\ups(v)}w(r),$$
which is the width of the rectangle $r$ divided by the weight of $v$. When $v\neq a$, summing this over all rectangles with top side on $v$ gives the total length of the horizontal segments on $v$, divided by $\ups(v)$. Similarly, the contribution of a rectangle $r$ with bottom side on $v$ to $\ndiv\mathbf F(v)$ is
$$-\frac{1}{\ups(v)}h(r)\frac{w(r)}{h(r)}=-\frac{1}{\ups(v)}w(r).$$
When $v\neq b$, summing this over all rectangles with bottom side on $v$ gives the negative of the total length of the horizontal segments on $v$, divided by $\ups(v)$. Therefore, for all $v\in \{a,b\}^c$, we have $\ndiv\mathbf F(v)=0$, proving part (i).

To show part (ii), note by Proposition~\ref{prop:cons} that it is equivalent to show that $\mathbf F$ lies in $\cut(E,\eps)$. The latter follows because $\mathbf F= \grad y$, where $y\in L^2(V,\ups)$ describes the $y$-coordinates of the horizontal lines in the plane.
\end{proof}

The next theorem gives an interlacing result using rectangular tilings and the edge Laplacian.

\begin{theorem} \label{the:rect}
Let $(G,\ups,\eps)$ be a weighted multigraph and let $e=ab$ be an edge. Let $R$ be a rectangle satisfying $\eps(e) = \frac{w(R)}{h(R)}$, and consider any rectangular tiling of $R$. Define $(G',\ups',\eps')$ to be the weighted graph obtained by replacing the weighted edge $e$ with the edge-weighted graph associated with the tiling, identifying $a$ and $b$ with the vertices corresponding to the bottom and top sides of $R$, respectively, and letting the weights of the new vertices be arbitrary. Let $\lambda_1\le\dots\le\lambda_n$ and $\mu_1\le\dots\le\mu_{n+k}$ denote the weighted Laplacian eigenvalues of $(G,\ups,\eps)$ and $(G',\ups',\eps')$, respectively. Then
$$\mu_i\le\lambda_i\le\mu_{i+k}$$
for all $1\le i \le n$.
\end{theorem}

\begin{proof}
It suffices to prove the inequalities for the nonzero eigenvalues, so we will prove the interlacing for the eigenvalues of the edge Laplacian. Let $G'=(V',E')$ and write $E'=E_1\cup E_2$, where $E_1$ consists of the edges of $E'$ from $E$, and $E_2$ consists of the edges of the graph associated with the tiling of $R$.

Define $\iota:L^2(E,\eps)\to L^2(E',\eps')$ by
$$\iota \mathbf F(uv) =\mathbf F(uv)$$
if $uv\in E_1$, and
$$\iota \mathbf F(r) = \frac{ \mathbf F(ab)\cdot\mathbf v_{\vec{ab}}}{h(R)} h(r)\mathbf v_{\vec{uv}}$$
if $r=uv\in E_2$ with $u$ and $v$ corresponding to the bottom and top sides of $r$, respectively.

First, we will check that $\iota$ is an isometry. Note that it sends an orthogonal basis for $L^2(E,\eps)$ of characteristic functions of directed edges to orthogonal functions in $L^2(E',\eps')$. It also preserves squared norms of characteristic functions of directed edges in $\dire E_1$. It remains to show that the squared norm of $\indb_{\vec{ab}}$ in $L^2(E,\eps)$ is equal to the squared norm of $\iota\indb_{\vec{ab}}$ in $L^2(E',\eps')$, which follows because
$$\int_E \lVert\indb_{\vec{ab}}\rVert^2d\eps = \frac{w(R)}{h(R)}=\frac{1}{h(R)^2} \sum_{r\in E_2}h(r)^2\frac{w(r)}{h(r)}=\int_{E'}\lVert\iota\indb_{\vec{ab}}\rVert^2d\eps',$$
where the second equality holds because the left expression is equal to $\frac{1}{h(R)^2}$ of the area of $R$, while the right expression is equal to $\frac{1}{h(R)^2}$ of the sum of the areas of the rectangles $r$ used to tile $R$. So $\iota$ is an isometry.

Moreover, by Lemma~\ref{lem:kirchhoff}(i) and its proof, we have $$\ndiv\iota \indb_{\vec{ab}}(v) = 0$$
for all vertices $v\neq a,b$, while
$$\ndiv\iota\indb_{\vec{ab}}(a)=-\frac{1}{\ups(a)h(R)}w(R)=\ndiv\indb_{\vec{ab}}(a)$$
and
$$\ndiv\iota\indb_{\vec{ab}}(b)=\frac{1}{\ups(b)h(R)}w(R)=\ndiv\indb_{\vec{ab}}(b),$$
since $w(R)$ is the total length of the horizontal segments on the line corresponding to $a$ and on the line corresponding to $b$. It follows that for all $\mathbf F\in L^2(E,\eps)$ and all $v\in V$, we have
$$\ndiv \iota\mathbf F(v)=\ndiv\mathbf F(v)$$
and for all $v\in V'\setminus V$,
$$\ndiv\iota\mathbf F(v)=0.$$

Next, we will show that $\iota$ restricts to a well-defined map $\cut(E,\eps)\to\cut(E',\eps')$. To see that $\iota\mathbf F$ lies in $\cut(E',\eps')$ whenever $\mathbf F\in\cut(E,\eps)$, we first note that the restriction of $\iota\mathbf F$ on $(G'[E_1],\ups',\eps')$ is the restriction of $\mathbf F$ on $(G[E_1],\ups,\eps)$, which is conservative as $\mathbf F$ is conservative on $(G,\ups,\eps)$. The restriction of $\iota\mathbf F$ on $(G'[E_2],\ups',\eps')$ is also conservative, being a scaled multiple of a vector field shown to be conservative in Lemma~\ref{lem:kirchhoff}(ii). In particular, $\iota\mathbf F = \frac{\mathbf F(ab)\cdot\mathbf v_{\vec{ab}}}{h(R)}\grad y$ on $E_2$, where $y$ gives the $y$-coordinate of horizontal lines in the plane.

Let $\frac{C}{\eps'}$ be a simple closed curve on $(G',\ups',\eps')$. If $C$ uses only edges in $E_1$ or only edges in $E_2$, then
$$\int_{\frac{C}{\eps'}}\iota\mathbf F\cdot d\eps' = 0,$$
since $\iota\mathbf F$ is conservative when restricted to either $(G'[E_1],\ups',\eps')$ or $(G'[E_2],\ups',\eps')$. Otherwise, if $\frac{C}{\eps'}$ uses edges in both $E_1$ and $E_2$, then its edges can be split into those of two curves, $\frac{C_1}{\eps'}$ and $\frac{C_2}{\eps'}$ from $a$ to $b$ and from $b$ to $a$, respectively, each using only edges from one of $E_1$ or $E_2$.

If $\frac{C_1}{\eps'}$ uses only edges in $E_1$, then
$$\int_{\frac{C_1}{\eps'}}\iota\mathbf F\cdot d\eps' = \int_{\frac{C_1}{\eps}}\mathbf F\cdot d\eps = \int_{\frac{\vec {ab}}{\eps}}\mathbf F\cdot d\eps = \mathbf F(ab)\cdot\mathbf v_{\vec{ab}},$$
where $\frac{\vec{ab}}{\eps}$ denotes the curve from $a$ to $b$ in $(G,\ups,\eps)$ via the directed edge $\vec{ab}$, and the second equality is by Proposition~\ref{prop:cons}. Otherwise, if $\frac{C_1}{\eps'}$ uses only edges in $E_2$, then
$$\int_{\frac{C_1}{\eps'}}\iota\mathbf F\cdot d\eps' = \int_{\frac{C_1}{\eps'}}\frac{\mathbf F(ab)\cdot\mathbf v_{\vec{ab}}}{h(R)}\grad y\cdot d\eps'=\frac{\mathbf F(ab)\cdot\mathbf v_{\vec{ab}}}{h(R)}(y(b)-y(a))=\mathbf F(ab)\cdot\mathbf v_{\vec{ab}},$$
with the second equality by Proposition~\ref{prop:grad} and the third equality because the difference between the $y$-coordinates of the top and bottom sides of $R$ is the height of $R$. Similarly, $\int_{\frac{C_2}{\eps'}}\iota\mathbf F\cdot d\eps'=-\mathbf F(ab)\cdot\mathbf v_{\vec{ab}},$ regardless of whether $\frac{C_2}{\eps}$ only uses edges in $E_1$ or in $E_2$. Hence, $\int_{\frac{C}{\eps'}}\iota\mathbf F\cdot d\eps'=0$ for all simple closed curves $\frac{C}{\eps'}$, and so $\iota\mathbf F\in \cut(E',\eps')$ by Proposition~\ref{prop:cons}.

Now consider the restrictions of the edge Laplacians to the cut spaces, i.e. $K_{(G,\ups,\eps)}:\cut(E,\eps)\to\cut(E,\eps)$ and $K_{(G',\ups',\eps')}:\cut(E',\eps')\to\cut(E',\eps')$, which are well-defined and self-adjoint since the cut space is the image of the edge Laplacian, which is the span of its eigenfunctions with nonzero eigenvalues. By Remark~\ref{rem:preserve}, since
$$\langle K_{(G,\ups,\eps)}\mathbf F,\mathbf F\rangle = \langle \ndiv\mathbf F,\ndiv\mathbf F\rangle = \langle\ndiv\iota\mathbf F,\ndiv\iota\mathbf F\rangle = \langle K_{(G',\ups',\eps')}\iota\mathbf F,\iota\mathbf F\rangle$$
for all $\mathbf F\in\cut(E,\eps)$, it follows that $K_{(G,\ups,\eps)}=\iota^*K_{(G',\ups',\eps')}\iota$. Then Theorem~\ref{the:cauchy}, the Cauchy interlacing theorem, gives a series of interlacing inequalities for the nonzero eigenvalues of the edge Laplacians. Since $G$ and $G'$ have the same number of connected components, $(G',\ups',\eps')$ has $k$ more nonzero edge Laplacian eigenvalues than $(G,\ups,\eps)$, and so
$$\mu_i\le\lambda_i\le\mu_{i+k}$$
for all $1\le i \le n$.
\end{proof}

\begin{remark}
Brooks, Smith, Stone and Tutte \cite{BSST} were primarily interested in square tilings of squares, especially \textit{perfect squarings}, where only squares of distinct side lengths are used. When Theorem~\ref{the:rect} is applied with a square tiling of a square and all edge weights of $(G,\ups,\eps)$ are $1$, then all edge weights of $(G',\ups',\eps')$ are also $1$. If moreover the squaring is perfect and $G$ is simple, then $G'$ is guaranteed to also be simple (as multiple edges between a pair of vertices would require two squares of equal height in the tiling).
\end{remark}

\begin{remark}
Given a rectangular tiling, there may be multiple edge-weighted graphs which can be constructed from the tiling that satisfy Kirchhoff's laws as in Lemma~\ref{lem:kirchhoff}. For example, one can alternatively take the vertex set to be the set of horizontal segments, as is done in \cite{BSST}. Then a version of Theorem~\ref{the:rect} still holds, as the proof follows through, although the bounds on the weighted Laplacian eigenvalues of $(G,\ups,\eps)$ are looser, as the interlacing result can also be deduced by applying Theorem~\ref{cor:quotient} and the stated version of Theorem~\ref{the:rect}.
\end{remark}

Continuing the analogy with electrical networks, we give the following definition from the theory of network reduction. A \textit{Kron reduction} of a weighted graph $(G,\ups,\eps)$ over $S\subseteq V$ not containing any connected components is a weighted graph $(G/S,\ups,\eps/S)$ with vertex set $S^c$ and vertex weights given by $\ups$, so that for any $f\in L^2(S^c,\ups)$ and $u\in S^c$, we have
$$L_{(G/S,\ups,\eps/S)}f(u) = L_{(G,\ups,\eps)}\tilde f(u),$$
where $\tilde f\in L^2(V,\ups)$ is an extension of $f$ satisfying
$$L_{(G,\ups,\eps)}\tilde f(v)=0$$
for all $v\in S$.

If we write the matrix of $L_{(G,\ups,\eps)}$ in the basis of characteristic functions of vertices as the block matrix $$L_{(G,\ups,\eps)}=\begin{pmatrix}
L_{(G,\ups,\eps)}^{\hat S}& B\\
C&D
\end{pmatrix},$$
where the first rows and columns are ordered to correspond to $S^c$, we can derive the matrix of $L_{(G/S,\ups,\eps/S)}$. Note that the values of $D\tilde f$ on $S$ must sum with $Cf$ to equal zero. Since $S$ does not contain any connected component, there is a path from each vertex in $V$ to a vertex of $S^c$, and so there exists an $S^c$-rooted spanning forest of $G$. By Theorem~\ref{the:mtree}, the matrix $D$ is invertible, and the values of $\tilde f$ on $S$ are given exactly by $-D^{-1}Cf$. Then
$$L_{(G/S,\ups,\eps/S)}f = \begin{pmatrix}L_{(G,\ups,\eps)}^{\hat S} & B\end{pmatrix}\begin{pmatrix}f\\ -D^{-1}Cf\end{pmatrix} = L_{(G,\ups,\eps)}^{\hat S}f - BD^{-1}Cf$$
for all $f\in L^2(S^c,\ups)$, and so $L_{(G/S,\ups,\eps/S)}$ must identically equal $L_{(G,\ups,\eps)}^{\hat S} - BD^{-1}C$. For general matrices, this is also known as a \textit{Schur complement}.

Since vertex weights are given by $\ups$, from the matrix of $L_{(G/S,\ups,\eps/S)}$, we can also deduce the total edge weight between any pair of distinct vertices. Therefore, if a Kron reduction $(G/S,\ups,\eps/S)$ exists, it is unique up to the operations in Lemma~\ref{lem:simplify}(ii) and (iii).

One special case of a Kron reduction is known as the \textit{star-mesh transform} in circuit analysis. Given a loopless nonisolated vertex $v$, the weighted graph $(G/v,\ups,\eps/v)$ is obtained by deleting $v$ and all incident edges, and adding for each pair of edges $u_1v$ and $u_2v$ with $u_1\neq u_2$ a new edge $u_1u_2$ with weight given by
$$(\eps/v)(u_1u_2)=\frac{\eps(u_1v)\eps(u_2v)}{d_{(G,\eps)}(v)}.$$
The matrix of $L_{(G/v,\ups,\eps/v)}$ can be checked to coincide with our formula for the Kron reduction over $v$. Note then that the existence of a Kron reduction over $S\subseteq V$ containing no connected components can be deduced by repeated applications of the star-mesh transform, and that the number of connected components remains constant after a Kron reduction.

\begin{example}
Let $(G,\ups,\eps)$ be the weighted graph depicted on the left, and let $v$ be the vertex of weight $5$ in $(G,\ups,\eps)$. Then $(G/v,\ups,\eps/v)$, obtained from $(G,\ups,\eps)$ by applying the star-mesh transform at vertex $v$, is the weighted graph depicted on the right.
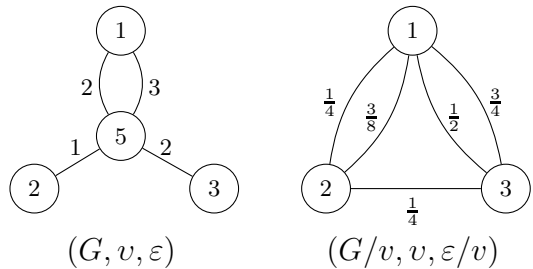
\begin{center}
\begin{tikzpicture}
\node[shape=circle, draw=black](1) at (4.75,-1.05) {\scriptsize 2};
\node[shape=circle, draw=black](2) at (7.14,-1.05) {\scriptsize 3};
\node[shape=circle,  draw=black] (3) at (5.9,-.35) {\scriptsize 5};
\node[shape=circle,  draw=black] (4) at (5.9,1.05) {\scriptsize 1};
\path[-] (1) edge (3);
\path[-, bend right] (3) edge (4);
\path[-, bend left] (3) edge (4);
\path[-] (3) edge (2);
\node[] at (5.45,.3) {\scriptsize 2};
\node[] at (6.35,.3) {\scriptsize 3};
\node[] at (6.5,-.5) {\scriptsize 2};
\node[] at (5.3,-.5) {\scriptsize 1};
\node[] at (5.9,-1.9) { $(G,\ups,\eps)$};
\end{tikzpicture}\quad\quad
\begin{tikzpicture}
\node[shape=circle, draw=black](1) at (4.75,-1.05) {\scriptsize 2};
\node[shape=circle, draw=black](2) at (7.14,-1.05) {\scriptsize 3};
\node[shape=circle,  draw=black] (4) at (5.9,1.05) {\scriptsize 1};
\path[-, bend right=20] (2) edge (4);
\path[-, bend left=20] (2) edge (4);
\path[-, bend right=20] (1) edge (4);
\path[-, bend left=20] (1) edge (4);
\path[-] (1) edge (2);
\node[] at (4.8,0.1) {\scriptsize $\frac{1}{4}$};
\node[] at (5.35,-0.1) {\scriptsize $\frac{3}{8}$};
\node[] at (7,0.1) {\scriptsize $\frac{3}{4}$};
\node[] at (6.45,-0.1) {\scriptsize $\frac{1}{2}$};
\node[] at (5.9,-1.35) {\scriptsize $\frac{1}{4}$};
\node[] at (5.9,-1.9) { $(G/v,\ups,\eps/v)$};
\end{tikzpicture}
\end{center}
\end{example}

We will now prove an interlacing result on the weighted Laplacian eigenvalues of a weighted graph and its Kron reduction. It is also known more generally that the eigenvalues of a semidefinite real symmetric matrix interlace those of a Schur complement \cite[Theorem 5]{Smith}, which can be extended to apply to the weighted Laplacian, despite it not necessarily being symmetric. We offer a different proof using the edge Laplacian. Note that Theorem~\ref{the:rect} is a special case of the following theorem in which the Kron reduction has a geometric interpretation and is easy to compute.

\begin{theorem} \label{the:kron}
Let $(G,\ups,\eps)$ be a weighted multigraph and let $S\subseteq V$ not contain any connected components. Let $\lambda_1\le\dots\le \lambda_n$ and $\mu_1\le\dots\le \mu_{n-k}$ denote the weighted Laplacian eigenvalues of $(G,\ups,\eps)$ and its Kron reduction $(G/S,\ups,\eps/S)$, respectively. Then
$$\lambda_i\le\mu_i\le \lambda_{i+k}$$
for all $1\le i\le n-k$.
\end{theorem}

\begin{proof}
We will consider the restrictions $K_{(G,\ups,\eps)}: \cut(E,\eps)\to\cut(E,\eps)$ and $K_{(G/S,\ups,\eps/S)}:\cut(E/S,\eps/S)\to\cut(E/S,\eps/S)$, where $E/S$ denotes the edge set of $G/S$, of the edge Laplacians to the relevant cut spaces. 

Define $\iota:\cut(E/S,\eps/S)\to\cut(E,\eps)$ to be the linear map sending $\grad f\mapsto \grad\tilde f$ for each $f\in L^2(S^c,\ups)$. Then $\iota$ is an isometry, since for each $f,g\in L^2(S^c,\ups)$, we have
$$\langle\grad f,\grad g\rangle = \langle L_{(G/S,\ups,\eps/S)}f,g\rangle = \langle L_{(G,\ups,\eps)}\tilde f,\tilde g\rangle = \langle \grad \tilde f,\grad \tilde g\rangle,$$
recalling that $L_{(G,\ups,\eps)}\tilde f$ coincides with $L_{(G/S,\ups,\eps/S)}f$ on $S^c$, and otherwise equals zero on $S$, and that $\tilde g$ coincides with $g$ on $S^c$.

Note $\iota$ also preserves quadratic form, since for all $f,g\in L^2(S^c,\ups)$,
$$\langle K_{(G/S,\ups,\eps/S)}\grad f,\grad g\rangle = \langle L_{(G/S,\ups,\eps/S)}f,L_{(G/S,\ups,\eps/S)}g\rangle =  \langle L_{(G,\ups,\eps)}\tilde f,L_{(G,\ups,\eps)}\tilde g\rangle=\langle K_{(G,\ups,\eps)}\grad \tilde f,\grad\tilde g\rangle,$$
as $L_{(G,\ups,\eps)}\tilde f$ and $L_{(G,\ups,\eps)}\tilde g$ coincide with $L_{(G/S,\ups,\eps/S)}f$ and $L_{(G/S,\ups,\eps/S)}g$, respectively, on $S^c$, and otherwise equal zero on $S$. By Remark~\ref{rem:preserve}, $K_{(G/S,\ups,\eps/S)}=\iota^*K_{(G,\ups,\eps)}\iota$.

The Cauchy interlacing theorem, Theorem~\ref{the:cauchy}, gives interlacing inequalities on the nonzero eigenvalues of the edge Laplacians. Since $G/S$ has the same number of connected components as $G$, we conclude that
$$\lambda_i\le\mu_i\le \lambda_{i+k}$$
for all $1\le i\le n-k$.
\end{proof}

Interestingly, vertex reduction by the star-mesh transform appears to be analogous to edge contraction. In the former, the weighted Laplacian Rayleigh quotients for conservative vector fields on the reduced weighted graph may be realised as the weighted Laplacian Rayleigh quotients for conservative vector fields with zero divergence at the reduced vertex in the original weighted graph. In the latter, the weighted Laplacian Rayleigh quotients for scalar fields on the contracted weighted graph may be realised as the weighted Laplacian Rayleigh quotients for scalar fields with zero gradient at the contracted edge in the original weighted graph. It is natural to ask, then, if there exists a vertex analogue of the deletion-contraction relation in Theorem~\ref{the:delcon}.

Note in edge deletion that the total amount of edge weight between the endpoints of the deleted edge is modified. The vertex analogue then is to modify the weight of the relevant vertex. We now prove a vertex analogue of deletion-contraction, which we call \textit{addition-reduction}.

\begin{theorem}\label{the:addred}
Let $(G,\ups,\eps)$ be a weighted multigraph and let $v$ be a loopless nonisolated vertex. Let $\eta(v) \in \mathbb R_{>0}$. Then
$$\p_{(G,\ups,\eps)}(t) = \frac{\ups(v)}{\ups(v)+\eta(v)}\p_{(G,\ups+\eta,\eps)}(t) - \frac{\eta(v) }{\ups(v)+\eta(v)}d_{(G,\eps)}(v)\p_{(G/v,\ups,\eps/v)}(t).$$
\end{theorem}
\begin{proof}
Note that the polynomial $P_{(G,\ups,\eps)}(t)$, after fixing $G$ and all vertex and edge weights other than the weight of $v$, may be thought of as a degree $1$ polynomial in $\ups(v)$, with coefficients in $\mathbb R[t]$. Therefore, we can write
$$\p_{(G,\ups,\eps)}(t) = \ups(v)\p^v_{(G,\ups,\eps)}(t) + \p^{/ v}_{(G,\ups,\eps)}(t),$$
where $\p^{/ v}_{(G,\ups,\eps)}(t)$ is $\p_{(G,\ups,\eps)}(t)$, thought of as a polynomial in $\ups(v)$, evaluated at $\ups(v)=0$. Since $(G,\ups+\eta,\eps)$ differs from $(G,\ups,\eps)$ only on the weight of $v$, we also have
$$\p_{(G,\ups+\eta,\eps)}(t)=(\ups(v)+\eta(v))\p^v_{(G,\ups,\eps)}(t) + \p^{/ v}_{(G,\ups,\eps)}(t).$$
Therefore, to prove the theorem statement, it is sufficient to show that
$$\p^{/ v}_{(G,\ups,\eps)}(t) = -d_{(G,\eps)}(v)\p_{(G/v,\ups,\eps/v)}(t).$$

Writing $L_{(G,\eps)}$ as the block matrix
$$L_{(G,\eps)}=\begin{pmatrix}L_{(G,\eps)}^{\hat v}&B\\C&d_{(G,\eps)}(v)\end{pmatrix},$$
we find that
$$L_{(G/v,\eps/v)}=W_\ups^{\hat v} L_{(G/v,\ups,\eps/v)} = L_{(G,\eps)}^{\hat v} - \frac{1}{d_{(G,\eps)}(v)}BC,$$
where $W_\ups^{\hat v}$ denotes the principal submatrix of $W_\ups$ obtained by deleting the row and column corresponding to $v$.

The result then follows, as
\begin{align*}
\p^{/v}_{(G,\ups,\eps)}(t) &= \det\left(\begin{pmatrix}tW_\ups^{\hat v}-L_{(G,\eps)}^{\hat v}&-B\\-C&-d_{(G,\eps)}(v)\end{pmatrix}\begin{pmatrix}I&0\\-\frac{1}{d_{(G,\eps)}(v)}C&1\end{pmatrix}\right)\\
&=\det\begin{pmatrix}tW_{\ups}^{\hat v}-L_{(G/v,\eps/v)}&-B\\0&-d_{(G,\eps)}(v)\end{pmatrix} = -d_{(G,\eps)}(v)\p_{(G/v,\ups,\eps/v)}(t).
\end{align*}
\end{proof}

\begin{remark}
More generally, the proof of Theorem~\ref{the:addred} can be extended to show that for any square matrix and any row with nonzero diagonal entry, there is a nontrivial linear relation between its characteristic polynomial and the characteristic polynomials of the matrices obtained by scaling the row and taking the Schur complement over the row. Another way to recover the deletion-contraction relation in Theorem~\ref{the:delcon}, then, is by showing that the edge Laplacian of the contracted weighted graph may be realised as a Schur complement of the edge Laplacian of the original weighted graph.
\end{remark}

\section{Weighted spectral bounds}\label{sec:bounds}
In this section, we will give applications of weighted Laplacian eigenvalues to bounding various properties of a weighted graph. We will see that various results on combinatorial and normalised Laplacian eigenvalues can be proved through a unified approach using weighted Laplacian eigenvalues.

Given a weighted graph $(G,\ups,\eps)$ and $\emptyset\subsetneq S\subsetneq V$, define the \textit{isoperimetric ratio} of $S$ to be
$$\theta(S) = \frac{\eps(\grad S)}{\min\{\ups(S),\ups(S^c)\}},$$
which measures the quality of a weighted cut. Note that here we think of $\grad S$ as a set of undirected edges in $E$, in order for its measure under $\eps$ to be well-defined. The sparsest weighted cut attains an isoperimetric ratio of
$$\theta_{(G,\ups,\eps)}=\min_{\emptyset\subsetneq S\subsetneq V} \theta(S),$$
which we call the \textit{isoperimetric constant} of $(G,\ups,\eps)$ and is analogous to the Cheeger isoperimetric constant from Riemannian geometry.

The following theorem shows that given a scalar field in $L^2(V,\ups)$ orthogonal to $\ind_V$ with small Rayleigh quotient, one can construct a good cut. We adapt the proof in \cite[Theorem 2.2]{ChungBook}.

\begin{theorem}\label{the:sparsecut}
Let $(G,\ups,\eps)$ be a weighted multigraph with at least two vertices and let nonzero $f\in L^2(V,\ups)$ satisfy $\int_V fd\ups = 0$. Then there exists some $t\in\mathbb R$ for which $\emptyset \subsetneq S_t\subsetneq V$ and
$$\frac{\theta(S_t)^2}{2}\le \frac{\langle L_{(G,\ups,\eps)}f,f\rangle}{\langle f,f\rangle}\max_{v\in V}\frac{d_{(G,\eps)}(v)}{{\ups}(v)},$$ where $ S_t=\{v\in V\mid f(v)\le t\}.$
\end{theorem}
\begin{proof}
Let $r$ denote the least real number such that $\ups(S_r)\ge \frac{\ups(V)}{2}.$ Since $\int_V fd\ups = 0$, we have
$$\int_V |f|^2d\ups \le \int_V|f|^2d\ups + r^2 = \int_V|f-r|^2d\ups.$$
Let $\g = f-r$, and denote its positive and negative parts by $\g_+ = \max\{\g,0\}$ and $\g_-= \max\{-\g,0\}$, respectively. Then
\begin{align*}
\frac{\langle L_{(G,\ups,\eps)}f,f\rangle}{\langle f,f\rangle} &= \frac{\int_E \lVert\grad f\rVert^2d\eps}{\int_V |f|^2d\ups} \\
&\ge \frac{\int_E \lVert\grad \g\rVert^2d\eps}{\int_V |\g|^2d\ups}
\\&\ge \frac{\int_E \lVert\grad\g_+\rVert^2d\eps + \int_E\lVert\grad\g_-\rVert^2d\eps}{\int_V |\g_+|^2d\ups + \int_V |\g_-|^2d\ups}\ge \min \left\{\frac{\int_E \lVert\grad\g_+\rVert^2d\eps }{\int_V |\g_+|^2d\ups } ,\frac{ \int_E\lVert\grad\g_-\rVert^2d\eps}{\int_V |\g_-|^2d\ups} \right\},
\end{align*}
where in the last expression we take the minimum among the terms that are defined. We assume without loss of generality that this minimum is $\frac{\int_E \lVert\grad\g_+\rVert^2d\eps }{\int_V |\g_+|^2d\ups }$.

Note after extending the definition of $\frac{\lVert\grad \g_+^2\rVert^2}{\lVert\grad \g_+\rVert^2}$ where it is undefined, we have for each edge $uv\in E$ that
$$\frac{\lVert\grad \g_+^2(uv)\rVert^2}{\lVert\grad \g_+(uv)\rVert^2}=|\g_+(u)+\g_+(v)|^2\le 2(|\g_+(u)|^2+|\g_+(v)|^2).$$
Therefore,
$$\int_E \frac{\lVert\grad \g_+^2\rVert^2}{\lVert\grad \g_+\rVert^2}d\eps \le 2\int_V|\g_+|^2dd_{(G,\eps)},$$
and so
$$\frac{\int_E \lVert \grad \g_+\rVert^2d\eps}{\int_V |\g_+|^2d\ups} \ge \frac{\left(\int_E \lVert \grad \g_+\rVert^2d\eps\right)\left(\int_E\frac{\lVert\grad \g_+^2\rVert^2}{\lVert\grad \g_+\rVert^2}d\eps\right)}{\left(\int_V |\g_+|^2d\ups\right)\left(2\int_V |\g_+|^2dd_{(G,\eps)}\right)} \ge \frac{\left(\int_E\lVert\grad \g_+^2\rVert d\eps\right)^2}{2 \left(\int_V |\g_+|^2d\ups\right)^2\max_{v\in V}\frac{d_{(G,\eps)}(v)}{\ups(v)}},$$
where in the second inequality we apply Cauchy-Schwarz to the numerator.

Define
$$\gamma_{f} = \min_{\emptyset\subsetneq S_t\subsetneq V}\theta(S_t),$$
and choose real numbers $t_1<\dots<t_k$ such that all distinct level sets of $\g_+$ (and hence of $\g_+^2$) are given by $L_1=S_{t_1}$, $L_2=S_{t_2}\setminus S_{t_1}$, \dots, $L_{k-1}= S_{t_k}\setminus S_{t_{k-1}}$. Then
\begin{align*}
\int _E \lVert\grad\g_+^2\rVert d\eps &=\sum_{i=1}^{k-1} (|\g_+(L_{i+1})|^2-|\g_+(L_i)|^2) \eps(\grad S_{t_i})\\
&\ge \gamma_{f}\sum_{i=1}^{k-1}(|\g_+(L_{i+1})|^2-|\g_+(L_i)|^2) \ups(S_{t_i}^c)\\
&=\gamma_f\sum_{i=1}^k |\g_+(L_i)|^2\ups(L_i) = \gamma_f\int_V |\g_+|^2 d\ups,
\end{align*}
where the first line follows by counting, and the second line by definition of $\gamma_f$ and noting that each $\ups(S_{t_i})\ge \ups(S_r)\ge \frac{\ups(V)}{2}$.

Putting everything together, we obtain
$$\frac{\langle L_{(G,\ups,\eps)}f,f\rangle}{\langle f,f\rangle}\ge \frac{\left(\gamma_f \int_V|\g_+|^2d\ups\right)^2}{2\left(\int_V |\g_+|^2d\ups \right)^2\max_{v\in V}\frac{d_{(G,\eps)}(v)}{\ups(v)}}.$$
Rearranging gives the desired inequality
$$\frac{\gamma_f^2}{2} \le \frac{\langle L_{(G,\ups,\eps)}f,f\rangle}{\langle f,f\rangle}\max_{v\in V}\frac{d_{(G,\eps)}(v)}{\ups(v)}.$$
\end{proof}

As a corollary, we can show that a weighted graph has a good cut if and only if its second weighted Laplacian eigenvalue $\lambda_2$ is small. The harder direction of Corollary~\ref{cor:cheeger} was also proved by Friedman and Tillich, who also studied graph Laplacians with general vertex and edge measures, in \cite[Theorem 5.2]{FT}, which generalises \cite[Theorem 2.3]{Dodziuk} and \cite[Theorem 2.2]{ChungBook} for the combinatorial and normalised Laplacians, respectively.

\begin{corollary}[Weighted Cheeger inequality] \label{cor:cheeger} Let $(G,\ups,\eps)$ be a weighted multigraph with at least two vertices and let $\lambda_2$ be its second smallest weighted Laplacian eigenvalue. Then
$$\frac{\lambda_2}{2}\le \theta_{(G,\ups,\eps)}\le \sqrt{2\lambda_2 \max_{v\in V}\frac{d_{(G,\eps)}(v)}{{\ups}(v)}}$$
\end{corollary}
\begin{proof}
Let $\emptyset\subsetneq S \subsetneq V$ attain an isoperimetric ratio of $\theta_{(G,\ups,\eps)}$. Consider the scalar field $f=\ind_S - \frac{\ups(S)}{\ups(V)}\ind_V$. We note that
$$\langle L_{(G,\ups,\eps)}f,f\rangle =\int_E\lVert\grad\ind_S\rVert^2d\eps = \eps(\grad S),$$
where the first equality follows because $\ind_V$ lies in the kernel of $\grad$. Additionally, we can compute from the definition of $f$ that
$$\langle f,f\rangle = \int_S\left|1-\frac{\ups(S)}{\ups(V)}\right|^2d\ups + \int_{S^c}\left|\frac{\ups(S)}{\ups(V)}\right|^2d\ups = \frac{\ups(S)\ups(S^c)}{\ups(V)}.$$
Note that
$$\ups(S)\ups(S^c) = \min\{\ups(S),\ups(S^c)\} \max\{\ups(S),\ups(S^c)\}$$
with
$$\frac{\max\{\ups(S),\ups(S^c)\}}{\ups(V)}\ge \frac{1}{2},$$
and so 
$$\frac{\lambda_2}{2}\le\frac{\langle L_{(G,\ups,\eps)}f,f\rangle}{2\langle f,f\rangle} \le \frac{\eps(\grad S)}{\min\{\ups(S),\ups(S^c)\}}=\theta_{(G,\ups,\eps)}.$$ The left inequality above is a consequence of Lemma~\ref{lem:rq} applied to $L_{(G,\ups,\eps)}$ restricted to the orthogonal complement of $\ind_V$, since nonzero $f\in L^2(V,\ups)$ satisfies
$$\int_V fd\ups = \int_S d\ups - \frac{\ups(S)}{\ups(V)}\int_V d\ups = \ups(S)-\frac{\ups(S)}{\ups(V)}\ups(V)=0.$$

For the other direction, let nonzero $f_2\in L^2(V,\ups)$ be an eigenfunction of $L_{(G,\ups,\eps)}$ orthogonal to $1_V$ with eigenvalue $\lambda_2$, and apply Theorem~\ref{the:sparsecut}.
\end{proof}

We next prove a bound on the weight of an \textit{independent set}, a subset $S\subseteq V$ with no two vertices of $S$ joined by an edge, using the largest weighted Laplacian eigenvalue $\lambda_n$ of a weighted graph.

\begin{theorem}\label{the:indep}
Let $(G,\ups,\eps)$ be a weighted multigraph with at least one nonloop edge and let $\lambda_n$ be its largest weighted Laplacian eigenvalue. Then for any nonempty independent set $S$, we have
$$\ups(S)\le \ups(V)\frac{\lambda_n-\frac{d_{(G,\eps)}(S)}{\ups(S)}}{\lambda_n}.$$
\end{theorem}
\begin{proof}
Consider the scalar field $f=\ind_S - \frac{\ups(S)}{\ups(V)}\ind_{V}\in L^2(V,\ups)$. By the same calculations as in the proof of Theorem~\ref{cor:cheeger},
$$\langle L_{(G,\ups,\eps)}f,f\rangle =\eps(\grad S)=d_{(G,\eps)}(S),$$
where the right equality is because $S$ is an independent set and so $\grad S$ consists of all edges incident to any vertex of $S$, and
$$\langle f,f\rangle =\frac{\ups(S)\ups(S^c)}{\ups(V)}=\ups(S)\left(1-\frac{\ups(S)}{\ups(V)}\right).$$

Therefore, the Rayleigh quotient of $f$ satisfies
$$\frac{\langle L_{(G,\ups,\eps)}f,f\rangle}{\langle f,f\rangle}=\frac{d_{(G,\eps)}(S)}{\ups(S)\left(1-\frac{\ups(S)}{\ups(V)}\right)} \le \lambda_n,$$
by Lemma~\ref{lem:rq}. Rearranging gives the desired inequality
$$\ups(S)\le \ups(V)\frac{\lambda_n - \frac{d_{(G,\eps)}(S)}{\ups(S)}}{\lambda_n}.$$
\end{proof}

\begin{remark}
Lower bounding the quantity $\frac{d_{(G,\eps)}(S)}{\ups(S)}$ by $\min_{v\in V}\frac{d_{(G,\eps)}(v)}{\ups(v)}$ gives a weaker but perhaps more practical bound on the weight of an independent set, without the dependence on $S$ on the right hand side.
\end{remark}

When $L_{(G,\ups,\eps)}$ is the combinatorial Laplacian, we obtain the following bound on the size of an independent set, which is also \cite[Corollary 3.5]{GN} of Godsil and Newman when all edge weights are $1$.

\begin{corollary}
Let $(G,\eps)$ be an edge-weighted graph with at least one nonloop edge and let $\lambda_n$ be its largest combinatorial Laplacian eigenvalue. Then for any nonempty independent set $S$, we have
$$|S|\le n\frac{\lambda_n-\bar d_{(G,\eps)}(S)}{\lambda_n},$$
where $\bar d_{(G,\eps)}(S)$ denotes the average degree of a vertex in $S$.
\end{corollary}

Additionally, one could think to apply Theorem~\ref{the:indep} when $L_{(G,\ups,\eps)}$ is the normalised Laplacian, which gives the following bound used by Chung \cite{ChungBook} in her proof of Corollary~\ref{cor:normcol}.

\begin{corollary} Let $(G,\eps)$ be an edge-weighted graph with no isolated vertices and at least one nonloop edge and let $\lambda_n$ be its largest normalised Laplacian eigenvalue. Then for any nonempty independent set $S$, we have
$$d_{(G,\eps)}(S) \le 2\eps(E) \frac{\lambda_n-1}{\lambda_n}.$$
\end{corollary}

We conclude this section with results on \textit{proper colourings} of a graph $G$, which are maps
$$\kappa: V\to \mathbb Z_{>0}$$
with the property that $\kappa(u)\neq\kappa(v)$ whenever $u,v$ are joined by an edge in $E$. More specifically, we will give spectral lower bounds on the \textit{chromatic number} $\chi(G)$ of a loopless graph $G$, which is the least $t\in \mathbb Z_{>0}$ for which there exists a \textit{proper $t$-colouring} of $G$, a proper colouring $\kappa$ with $\kappa(v)\in [t]$ for all vertices $v\in V$.

\begin{theorem}\label{the:chrom}
Let $(G,\ups,\eps)$ be a loopless weighted multigraph with at least one edge and let $\lambda_n$ be its largest weighted Laplacian eigenvalue. Then
$$\chi(G)\ge \frac{\lambda_n}{\lambda_n - \frac{2\eps(E)}{\ups(V)}}.$$
\end{theorem}
\begin{proof}
Fix a proper $\chi(G)$-colouring $\kappa: V\to [\chi(G)]$ of $G$. Define $x_\kappa,y_\kappa\in L^2(V,\ups)$ to be the scalar fields with $x_\kappa(v)= \cos \frac{2\pi \kappa(v)}{\chi(G)} $ and $y_\kappa(v) = \sin \frac{2\pi \kappa(v)}{\chi(G)}$ for all $v\in V$, and extend the notation similarly to other proper $\chi(G)$-colourings of $G$.

Let $\sigma \in S_{\chi(G)}$ be chosen uniformly at random. Then
\begin{align*}
\ex\frac{\langle L_{(G,\ups,\eps)}x_{\sigma\kappa},x_{\sigma\kappa}\rangle+\langle L_{(G,\ups,\eps)}y_{\sigma\kappa},y_{\sigma\kappa}\rangle}{\langle x_{\sigma\kappa},x_{\sigma\kappa}\rangle+\langle y_{\sigma\kappa},y_{\sigma\kappa}\rangle}&=\ex \frac{\int_E\lVert\grad x_{\sigma\kappa}\rVert^2d\eps+\int_E\lVert\grad y_{\sigma\kappa}\rVert^2d\eps}{\int_V |x_{\sigma\kappa}|^2d\ups+\int_V |y_{\sigma\kappa}|^2d\ups}\\ &=\frac{\ex[\int_E\lVert\grad x_{\sigma\kappa}\rVert^2d\eps+\int_E\lVert\grad y_{\sigma\kappa}\rVert^2d\eps]}{\ups(V)}\\
&=\frac{\int_E\ex[\lVert\grad x_{\sigma\kappa}\rVert^2+\lVert\grad y_{\sigma\kappa}\rVert^2]d\eps}{\ups(V)},
\end{align*}
where the second line follows since $\cos^2 \frac{2\pi\sigma(\kappa(v))}{\chi(G)}+\sin^2 \frac{2\pi\sigma(\kappa(v))}{\chi(G)}=1$ for all $\sigma\in S_{\chi(G)}$ and for all vertices $v\in V$, and the third line follows by interchanging the order of expectation and integration.

Note for any edge $e\in E$ that the values of its endpoints under $\sigma\kappa$ can be any pair of distinct values in $[\chi(G)]$ with equal likelihood. Then, letting $\omega=\cos\frac{2\pi}{\chi(G)}+  i\sin\frac{2\pi}{\chi(G)}$,
\begin{align*}
    \ex[\lVert\grad x_{\sigma\kappa}(e)\rVert^2+\lVert\grad y_{\sigma\kappa}(e)\rVert^2] &= \frac{1}{\chi(G)(\chi(G)-1)}\sum_{\substack{1\le i \le j \le \chi(G)\\ i\neq j}} |\omega^i-\omega^j|^2\\
    &=\frac{1}{\chi(G)-1}\sum_{i=1}^{\chi(G)-1} |\omega^i-1|^2\\
    &= \frac{1}{\chi(G)-1}\sum_{i=1}^{\chi(G)-1}(2-\omega^i-\omega^{-i})\\
    &=\frac{2}{\chi(G)-1}\left(\chi(G)-1-\sum_{i=1}^{\chi(G)-1}\omega^i\right)= \frac{2\chi(G)}{\chi(G)-1}
\end{align*}
for every edge $e\in E$. Therefore, 
$$\ex \frac{\langle L_{(G,\ups,\eps)}x_{\sigma\kappa},x_{\sigma\kappa}\rangle+\langle L_{(G,\ups,\eps)}y_{\sigma\kappa},y_{\sigma\kappa}\rangle}{\langle x_{\sigma\kappa},x_{\sigma\kappa}\rangle+\langle y_{\sigma\kappa},y_{\sigma\kappa}\rangle}=\frac{2\eps(E)\chi(G)}{\ups(V)(\chi(G)-1)},$$
and so there exists some $\sigma\in S_{\chi(G)}$ for which 
\begin{align*}
\frac{2\eps(E)\chi(G)}{\ups(V)(\chi(G)-1)}&\le\frac{\langle L_{(G,\ups,\eps)}x_{\sigma\kappa},x_{\sigma\kappa}\rangle+\langle L_{(G,\ups,\eps)}y_{\sigma\kappa},y_{\sigma\kappa}\rangle}{\langle x_{\sigma\kappa},x_{\sigma\kappa}\rangle+\langle y_{\sigma\kappa},y_{\sigma\kappa}\rangle}\\
&\le \max\left\{\frac{\langle L_{(G,\ups,\eps)}x_{\sigma\kappa},x_{\sigma\kappa}\rangle}{\langle x_{\sigma\kappa},x_{\sigma\kappa}\rangle},\frac{\langle L_{(G,\ups,\eps)}y_{\sigma\kappa},y_{\sigma\kappa}\rangle}{\langle y_{\sigma\kappa},y_{\sigma\kappa}\rangle}\right\}\le\lambda_n,
\end{align*}
where on the second line we take the maximum among the terms that are defined and apply Lemma~\ref{lem:rq}. Rearranging gives the desired inequality
$$\chi(G)\ge \frac{\lambda_n}{\lambda_n - \frac{2\eps(E)}{\ups(V)}}.$$
\end{proof}

When $L_{(G,\ups,\eps)}$ is the combinatorial Laplacian, we obtain the following, which was also found by Spielman in \cite[Section 3.7]{Spiel} in the case where all edge weights are $1$.

\begin{corollary}\label{cor:combcol}
Let $(G,\eps)$ be a loopless edge-weighted multigraph with at least one edge and let $\lambda_n$ be its largest combinatorial Laplacian eigenvalue. Then
$$\chi(G)\ge \frac{\lambda_n}{\lambda_n - \bar d_{(G,\eps)}},$$
where $\bar d_{(G,\eps)}$ denotes the average degree of a vertex.
\end{corollary}

Finally, if we apply our spectral bound on the chromatic number when $L_{(G,\ups,\eps)}$ is the normalised Laplacian, we obtain Chung's \cite[Theorem 6.7]{ChungBook}, which also appears in the work of Coutinho, Grandsire and Passos in \cite[Lemma 3.3]{CGP} when all edge weights are $1$.

\begin{corollary}\cite[Theorem 6.7]{ChungBook}\label{cor:normcol}
Let $(G,\eps)$ be a loopless edge-weighted multigraph with no isolated vertices and let $\lambda_n$ be its largest normalised Laplacian eigenvalue. Then
$$\chi(G)\ge 1 + \frac{1}{\lambda_n - 1}.$$
\end{corollary}

\begin{remark}
Spielman's \cite{Spiel} proof of Corollary~\ref{cor:combcol} in the case of unit edge weights and Chung's \cite{ChungBook} proof of Corollary~\ref{cor:normcol} used Theorem~\ref{the:indep} applied to the combinatorial and normalised Laplacians, respectively. Coutinho, Grandsire and Passos in \cite{CGP} gave a proof of Corollary~\ref{cor:normcol} in the case of unit edge weights using interlacing.

Our proof of Theorem~\ref{the:chrom} is distinct from both approaches and demonstrates an application of the probabilistic method. Moreover, our proof shows something slightly stronger than the theorem statement, namely that there exists a nonzero scalar field of one of the forms $x_{\sigma\kappa},y_{\sigma\kappa}$ with Rayleigh quotient $\ge \frac{2\eps(E)\chi(G)}{\ups(V)(\chi(G)-1)}$.
\end{remark}

\section{New families of trees distinguished by the chromatic symmetric function}\label{sec:csf}
In our final section, we will use deletion-contraction to relate the weighted Laplacian characteristic polynomial to Stanley's chromatic symmetric function \cite{Stan95}, and use this connection to prove new cases of a conjecture of Stanley.

We begin by introducing a few new definitions. An \textit{integer partition} $\alpha=(\alpha_1,\dots,\alpha_{\ell(\alpha)})$ is a list of positive integers $\alpha_1\ge\dots\ge\alpha_{\ell(\alpha)}$ sorted in weakly decreasing order. If $\alpha_1+\dots+\alpha_{\ell(\alpha)}=n$, then we write $\alpha\vdash n$. $\Sym$, the \textit{algebra of symmetric functions}, may be realised as a subalgebra of $\mathbb R[[x_1,x_2,\dots]]$, where the variables $x_j$ commute, as follows. The \textit{$i$th power sum symmetric function} $p_i$ is defined by
$$p_i=\sum_j x_j^i.$$
Given a partition $\alpha=(\alpha_1,\dots,\alpha_{\ell(\alpha)})$, the \textit{power sum symmetric function} $p_\alpha$ is
$$p_\alpha=\prod_{i=1}^{\ell(\alpha)}p_{\alpha_i}.$$
Then $\Sym$ is the algebra spanned by the basis $\{p_\alpha\}_{\alpha\vdash n\ge0}$.

Stanley's chromatic symmetric function \cite{Stan95} was extended to vertex-weighted graphs with integer weights by Crew and Spirkl in \cite{CS}.

\begin{definition}
The \textit{chromatic symmetric function} of an integer vertex-weighted multigraph $(G,\ups)$ with vertex set $\{v_1,\dots,v_n\}$ is
$$X_{(G,\ups)}=\sum_{\kappa} x_{\kappa(v_1)}^{\ups(v_1)}\dots x_{\kappa(v_n)}^{\ups(v_n)},$$
where the sum is over all proper colourings $\kappa$ of $G$.
\end{definition}

The motivation of Crew and Spirkl to consider vertex-weighted graphs was to obtain a deletion-contraction recurrence, which we state next.

\begin{theorem}\cite[Lemma 2]{CS}\label{csfdc}
Let $(G,\ups)$ be an integer vertex-weighted multigraph and let $e$ be an edge. Then
$$X_{(G,\ups)}= X_{(G-e,\ups)}-X_{(G/e,\ups/e)}.$$
\end{theorem}

\begin{remark}
Due to the deletion-contraction relation, the chromatic symmetric function $X_{(G,\ups)}$ is, up to a sign, an evaluation of the $W$-polynomial of Noble and Welsh \cite[Section 2]{NW} for integer vertex-weighted graphs. Similarly, by the relations in Theorem~\ref{the:delcon} and Lemma~\ref{lem:simplify}(ii), the polynomial $P_{(G,\ups)}(t)$ for integer vertex-weighted graphs is also, up to a sign, an evaluation of the $W$-polynomial.

In the special case where the graph is unweighted, the $W$-polynomial is known as the $U$-polynomial. The $U$-polynomial, Brylawski's polychromate \cite{Brylawski} and a generalisation of Stanley's chromatic symmetric function known as the Tutte symmetric function \cite{Stan98} have been shown to be equivalent in \cite{NW, Sarmiento}. Because the combinatorial Laplacian characteristic polynomial of an unweighted graph is, up to a sign, an evaluation of the $U$-polynomial, results on the three equivalent graph polynomials can give results on the combinatorial Laplacian.

For example, the construction of graphs with equal Tutte symmetric function in \cite[Theorem 11]{ACSZ}, which generalises a construction of Brylawski of graphs with equal polychromate in \cite[Theorem 4.6]{Brylawski}, also constructs graphs with equal combinatorial Laplacian spectrum. Because it employs a deletion-contraction argument, it can also be modified to construct weighted graphs with equal weighted Laplacian eigenvalues, even when not all weights are equal to $1$.

Another application is to the graph reconstruction conjecture of Kelly and Ulam, which asks whether for all graphs $G$ on at least three vertices the graph can be determined, up to isomorphism, from the multiset of isomorphism classes of the graphs $G-v$ for $v\in V$; see \cite{Bondy} for a survey. One approach to the conjecture is to study which graph polynomials and invariants are \textit{reconstructible}, or can be deduced from the multiset of isomorphism classes of the $G-v$ for all graphs $G$ on at least three vertices. 

Tutte showed in \cite{TutRec} that various polynomials and invariants, including the adjacency characteristic polynomial and chromatic polynomial of a graph, are reconstructible. It turns out that the combinatorial Laplacian characteristic polynomial is also reconstructible, because the $U$-polynomial is reconstructible, e.g. by considering the spanning subgraph expansion of the $U$-polynomial in \cite[Proposition 5.1]{NW}, and applying \cite[Theorem 6.7]{TutRec} and \cite[Theorem 6.8]{TutRec} to reconstruct the terms corresponding to the disconnected and connected spanning subgraphs, respectively.
\end{remark}

Stanley's original definition in \cite[Definition 2.1]{Stan95} considered only unweighted graphs, which is the case where all vertex weights are equal to $1$. We say that an (unweighted) simple graph is \textit{distinguished by the chromatic symmetric function} if any other simple graph with the same chromatic symmetric function must be isomorphic. Stanley's tree isomorphism conjecture is as follows.

\begin{conjecture}\cite[Section 2]{Stan95} \label{conj:trees}
Unweighted trees are distinguished by the chromatic symmetric function.
\end{conjecture}

It is known that any simple graph with the same chromatic symmetric function as a tree must also be a tree, since the chromatic symmetric function specialises to the \textit{chromatic polynomial $\chi_G(t)$}, which counts the number of proper $t$-colourings $\kappa: G\to [t]$, at $x_1=\dots= x_t = 1$ and $x_j=0$ for $j>t$, and the chromatic polynomial of a simple graph $G$ is equal to $t(t-1)^{n-1}$ if and only if $G$ is a tree on $n$ vertices. Some families of trees known to be distinguished by the chromatic symmetric function include all caterpillars \cite{LS} and all starlike trees \cite{MMW}.

We will now show that for integer vertex-weighted forests, the weighted Laplacian characteristic polynomial can be recovered from the chromatic symmetric function.

\begin{theorem}\label{the:csf}
Let $(F,\ups)$ be an integer vertex-weighted forest. Then
$$\p_{(F,\ups)}(t) = \varphi(X_{(F,\ups)}),$$
where $\varphi$ is the algebra homomorphism
\begin{align*}
\varphi:\Sym&\to \mathbb R[t]\\
p_\alpha&\mapsto \left(\prod_{i=1}^{\ell(\alpha)} \alpha_i\right)t^{\ell(\alpha)}.
\end{align*}
\end{theorem}
\begin{proof}
We will proceed by induction on the number of edges of $F$. If $F$ has no edges and its vertex weights under $\ups$ are given in weakly decreasing order by the partition $\alpha$, then the chromatic symmetric function of $(F,\ups)$ is exactly the power sum symmetric function $p_\alpha$, e.g. by \cite[Section 3]{CS}. Moreover, since $F$ has no edges, the combinatorial Laplacian $L_F$ is just the zero matrix. Therefore,
$$P_{(F,\ups)}(t) = \det (t W_\ups)=\left(\prod_{v\in V}\ups(v)\right)t^{|F|} = \left(\prod_{i=1}^{\ell(\alpha)}\alpha_i\right)t^{\ell(\alpha)}=\varphi(X_{(F,\ups)}).$$

For the inductive step, assume $F$ has at least one edge $e$. It is not a loop, as $F$ is a forest. Since $(F-e,\ups)$ and $(F/e,\ups/e)$ are both integer vertex-weighted forests with strictly fewer edges, we have by the inductive hypothesis and the deletion-contraction relations in Theorem~\ref{the:delcon} (recalling that for vertex-weighted graphs we treat all edge weights as being $1$) and Theorem~\ref{csfdc} that
$$P_{(F,\ups)}(t) = P_{(F-e,\ups)}(t)-P_{(F/e,\ups/e)}(t) = \varphi(X_{(F-e,\ups)})-\varphi(X_{(F/e,\ups/e)}) = \varphi (X_{(F,\ups)}),$$
as claimed.
\end{proof}

An (unweighted) simple graph $G$ is said to be \textit{determined by its combinatorial Laplacian spectrum} if any other simple graph with the same combinatorial Laplacian spectrum must be isomorphic. From Theorem~\ref{the:csf}, we can deduce the following.

\begin{corollary}\label{cor:dist}
If an unweighted tree $T$ is determined by its combinatorial Laplacian spectrum, then it is also distinguished by the chromatic symmetric function.
\end{corollary}
\begin{proof}
Suppose $T'$ is another unweighted simple graph with the same chromatic symmetric function $X_T=X_{T'}$. Then $T'$ must be a tree. By Theorem~\ref{the:csf},
$$\p_{T}(t) = \varphi(X_T)=\varphi(X_{T'})=\p_{T'}(t),$$
so $T$ and $T'$ have the same combinatorial Laplacian spectrum. Since $T$ is determined by its combinatorial Laplacian spectrum, any such $T'$ must be isomorphic to $T$, and so $T$ is distinguished by the chromatic symmetric function.
\end{proof}

There has been a large body of work on trees determined by their combinatorial Laplacian spectra, partly due to the interest generated by the survey paper of van Dam and Haemers \cite{vDH}, which asked more generally, which graphs are determined by their spectra for various matrices associated with a graph. In their concluding remarks \cite[Section 8]{vDH}, they noted that resolving the question for general graphs seemed out of reach, and proposed the more tractable problem of classifying which trees are determined by their spectra.

Various families of caterpillars \cite{ACSLUR, Boulet, BZL, LiThesis, LL, SH, Stanic}, including paths \cite{vDH}, and all starlike trees \cite{OT} are known to be determined by their combinatorial Laplacian spectra. Some families of trees which are determined by their combinatorial Laplacian spectra but not previously known to be distinguished by the chromatic symmetric function are described next.

\begin{corollary}\label{cor:newtrees}
The following unweighted trees are distinguished by the chromatic symmetric function:
\begin{enumerate}
    \item[(i)] Trees $B^a_{n}$ with a central vertex adjacent to $a$ vertices, each of which is joined by pendant edges to $n-1$ other vertices, for $n=3$ with $a\ge 1$, as well as for $n\ge a^2\ge 1$.
    \item[(ii)] Trees $h_{m,n}^a$ obtained by joining by edges a vertex to an endpoint of each of three paths on $a$, $m$ and $n$ vertices, respectively, and adding a pendant edge to each vertex of the path on $a$ vertices.
    \item[(iii)] Trees $H_{m,n}^a$ obtained by adding a pendant edge to the vertex on the path on $a$ vertices furthest from the initial vertex in $h_{m,n}^a$, for $a,m,n\ge 1$.
    \item[(iv)] Trees $i_n$ obtained from a path on $n$ vertices by joining an endpoint of a unique path on $2$ vertices to each vertex of the path on $n$ vertices, as well as to both endpoints of the path on $n$ vertices (with repetition), for $n\ge 1$.
    \item[(v)] Trees $M_{m,n}^{a,b}$ obtained by joining by edges one endpoint of an edge to an endpoint of each of two paths on $a$ and $m$ vertices, and the other endpoint of the edge to an endpoint of each of two paths on $b$ and $n$ vertices, for $a\ge m\ge1$, $b\ge n\ge1$, $a\ge b$ and $(a,b,m,n)\neq(m+2n+1,m+n+1,m,n),(b+2n+2,b,b+n,n)$. 
    \item[(vi)] Trees $T_{\ell,m,n}^a$ obtained by joining by edges a vertex to an endpoint of each of three paths on $\ell$, $m$, and $n$ vertices, and adding two pendant edges to the opposite endpoints of the first $a$ paths, for $\ell,m,n\ge 1$ and $1\le a\le 3$.
    \item[(vii)] Trees $y_n$ obtained by adding a pendant edge to each vertex of degree $2$ in $i_n$, for $n\ge 1$.
    \item[(viii)] Trees $Y_n$ obtained by adding a pendant edge to each vertex of the initial path in $y_n$, for $n\ge 1$.
    \item[(iv)] Trees $Z_n$ obtained by adding a pendant edge to each vertex of degree $3$ in $Y_n$, for $n\ge 1$.
\end{enumerate}
\end{corollary}
\begin{proof}
The given families of trees are shown to be determined by their combinatorial Laplacian spectra in \cite[Theorem 11]{AAS}, \cite[Theorem 3.2.4]{LiThesis}, \cite[Theorem 3.2.1]{LiThesis}, \cite[Theorem 3.3.2]{LiThesis}, \cite[Theorem 3.1]{WHHL}, \cite[Theorem 3.2]{WH}, \cite[Theorem 3.4.2]{LiThesis}, \cite[Theorem 2.1]{ZB} and \cite[Theorem 3.4.4]{LiThesis}, respectively. By Corollary~\ref{cor:dist}, they are distinguished by the chromatic symmetric function.
\end{proof}

\begin{example}
By Corollary~\ref{cor:newtrees}, the following trees are all distinguished by the chromatic symmetric function.
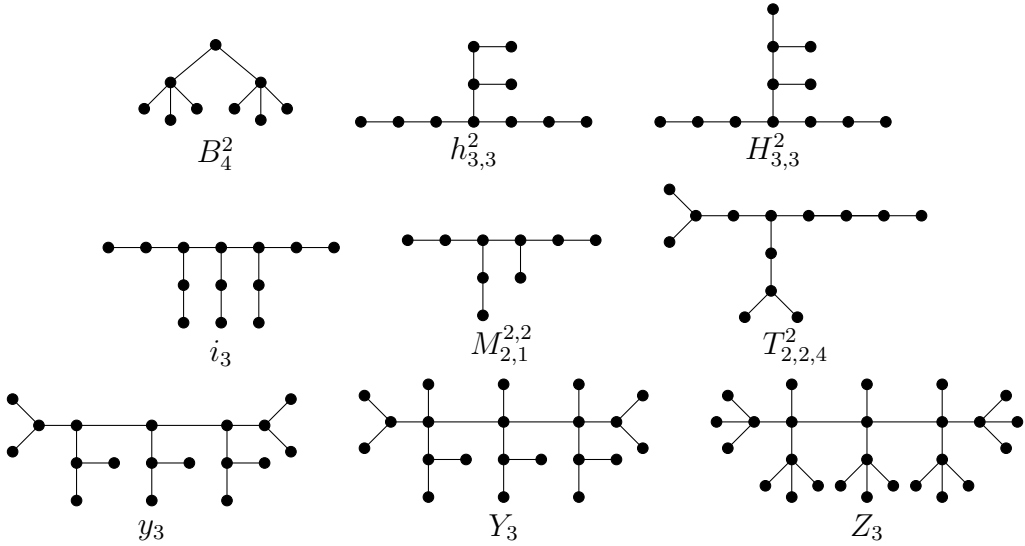
\begin{center}
\begin{tikzpicture}
\coordinate (1) at (0,0);
\coordinate (2) at (-.6,-.5);
\coordinate (3) at (.6,-.5);
\coordinate (4) at (.6,-1);
\coordinate (6) at (.95,-.85);
\coordinate (7) at (.25,-.85);
\coordinate (8) at (-.25,-.85);
\coordinate (9) at (-.95,-.85);
\coordinate (5) at (-.6,-1);
\filldraw [black] (1) circle (2pt);
\filldraw [black] (2) circle (2pt);
\filldraw [black] (3) circle (2pt);
\filldraw [black] (4) circle (2pt);
\filldraw [black] (5) circle (2pt);
\filldraw [black] (6) circle (2pt);
\filldraw [black] (7) circle (2pt);
\filldraw [black] (8) circle (2pt);
\filldraw [black] (9) circle (2pt);
\draw[-] (1) to (2);
\draw[-] (1) to (3);
\draw[-] (3)--(4);
\draw[-] (7)--(3);
\draw[-] (6)--(3);
\draw[-] (5)--(2);
\draw[-] (8)--(2);
\draw[-] (9)--(2);

\node[] at (0,-1.45) { $B^2_4$};
\end{tikzpicture}\quad\quad
\begin{tikzpicture}
\coordinate (1) at (0,0);
\coordinate (2) at (.5,0);
\coordinate (3) at (1,-0);
\coordinate (4) at (1.5,0);
\coordinate (5) at (2,0);
\coordinate (6) at (2.5,0);
\coordinate (7) at (3,0);
\coordinate (8) at (1.5,.5);
\coordinate (9) at (1.5,1);
\coordinate (10) at (2,.5);
\coordinate (11) at (2,1);
\filldraw [black] (1) circle (2pt);
\filldraw [black] (2) circle (2pt);
\filldraw [black] (3) circle (2pt);
\filldraw [black] (4) circle (2pt);
\filldraw [black] (5) circle (2pt);
\filldraw [black] (6) circle (2pt);
\filldraw [black] (7) circle (2pt);
\filldraw [black] (8) circle (2pt);
\filldraw [black] (9) circle (2pt);
\filldraw [black] (10) circle (2pt);
\filldraw [black] (11) circle (2pt);
\draw[-] (1) to (2);
\draw[-] (3) to (2);
\draw[-] (3) to (4);
\draw[-] (5) to (4);
\draw[-] (5) to (6);
\draw[-] (7) to (6);
\draw[-] (4) to (8);
\draw[-] (8) to (9);
\draw[-] (10) to (8);
\draw[-] (11) to (9);
\node[] at (1.5,-.4) { $h_{3,3}^2$};
\end{tikzpicture}\quad\quad
\begin{tikzpicture}
\coordinate (1) at (0,0);
\coordinate (2) at (.5,0);
\coordinate (3) at (1,-0);
\coordinate (4) at (1.5,0);
\coordinate (5) at (2,0);
\coordinate (6) at (2.5,0);
\coordinate (7) at (3,0);
\coordinate (8) at (1.5,.5);
\coordinate (9) at (1.5,1);
\coordinate (12) at (1.5,1.5);
\coordinate (10) at (2,.5);
\coordinate (11) at (2,1);
\filldraw [black] (1) circle (2pt);
\filldraw [black] (2) circle (2pt);
\filldraw [black] (3) circle (2pt);
\filldraw [black] (4) circle (2pt);
\filldraw [black] (5) circle (2pt);
\filldraw [black] (6) circle (2pt);
\filldraw [black] (7) circle (2pt);
\filldraw [black] (8) circle (2pt);
\filldraw [black] (9) circle (2pt);
\filldraw [black] (10) circle (2pt);
\filldraw [black] (11) circle (2pt);
\filldraw [black] (12) circle (2pt);
\draw[-] (1) to (2);
\draw[-] (3) to (2);
\draw[-] (3) to (4);
\draw[-] (5) to (4);
\draw[-] (5) to (6);
\draw[-] (7) to (6);
\draw[-] (4) to (8);
\draw[-] (8) to (9);
\draw[-] (12) to (9);
\draw[-] (10) to (8);
\draw[-] (11) to (9);
\node[] at (1.5,-.4) { $H_{3,3}^2$};
\end{tikzpicture}

\begin{tikzpicture}
\coordinate (1) at (0,0);
\coordinate (2) at (.5,0);
\coordinate (3) at (1,-0);
\coordinate (4) at (1.5,0);
\coordinate (5) at (2,0);
\coordinate (6) at (2.5,0);
\coordinate (7) at (3,0);
\coordinate (8) at (1.5,-.5);
\coordinate (9) at (1.5,-1);
\coordinate (10) at (2,-.5);
\coordinate (11) at (2,-1);
\coordinate (12) at (1,-.5);
\coordinate (13) at (1,-1);
\filldraw [black] (1) circle (2pt);
\filldraw [black] (2) circle (2pt);
\filldraw [black] (3) circle (2pt);
\filldraw [black] (4) circle (2pt);
\filldraw [black] (5) circle (2pt);
\filldraw [black] (6) circle (2pt);
\filldraw [black] (7) circle (2pt);
\filldraw [black] (8) circle (2pt);
\filldraw [black] (9) circle (2pt);
\filldraw [black] (10) circle (2pt);
\filldraw [black] (11) circle (2pt);
\filldraw [black] (12) circle (2pt);
\filldraw [black] (13) circle (2pt);
\draw[-] (1) to (2);
\draw[-] (3) to (2);
\draw[-] (3) to (4);
\draw[-] (5) to (4);
\draw[-] (5) to (6);
\draw[-] (7) to (6);
\draw[-] (4) to (8);
\draw[-] (8) to (9);
\draw[-] (5) to (10);
\draw[-] (10) to (11);
\draw[-] (3) to (12);
\draw[-] (13) to (12);
\node[] at (1.5,-1.4) { $i_3$};
\end{tikzpicture}\quad\quad
\begin{tikzpicture}
\coordinate (1) at (0,0);
\coordinate (2) at (.5,0);
\coordinate (3) at (1,-0);
\coordinate (4) at (1.5,0);
\coordinate (5) at (2,0);
\coordinate (6) at (2.5,0);
\coordinate (8) at (1.5,-.5);
\coordinate (12) at (1,-.5);
\coordinate (13) at (1,-1);
\filldraw [black] (1) circle (2pt);
\filldraw [black] (2) circle (2pt);
\filldraw [black] (3) circle (2pt);
\filldraw [black] (4) circle (2pt);
\filldraw [black] (5) circle (2pt);
\filldraw [black] (6) circle (2pt);
\filldraw [black] (8) circle (2pt);
\filldraw [black] (12) circle (2pt);
\filldraw [black] (13) circle (2pt);
\draw[-] (1) to (2);
\draw[-] (3) to (2);
\draw[-] (3) to (4);
\draw[-] (5) to (4);
\draw[-] (5) to (6);
\draw[-] (4) to (8);
\draw[-] (3) to (12);
\draw[-] (13) to (12);
\node[] at (1.25,-1.4) { $M_{2,1}^{2,2}$};
\end{tikzpicture}\quad\quad
\begin{tikzpicture}
\coordinate (1) at (0,0);
\coordinate (2) at (.5,0);
\coordinate (3) at (1,-0);
\coordinate (4) at (1.5,0);
\coordinate (5) at (2,0);
\coordinate (6) at (2.5,0);
\coordinate (8) at (3,0);
\coordinate (7) at (-.35,-.35);
\coordinate (9) at (-.35,.35);
\coordinate (10) at (1.35,-1.35);
\coordinate (11) at (.65,-1.35);
\coordinate (12) at (1,-.5);
\coordinate (13) at (1,-1);
\filldraw [black] (1) circle (2pt);
\filldraw [black] (2) circle (2pt);
\filldraw [black] (3) circle (2pt);
\filldraw [black] (4) circle (2pt);
\filldraw [black] (5) circle (2pt);
\filldraw [black] (6) circle (2pt);
\filldraw [black] (8) circle (2pt);
\filldraw [black] (7) circle (2pt);
\filldraw [black] (9) circle (2pt);
\filldraw [black] (12) circle (2pt);
\filldraw [black] (13) circle (2pt);
\filldraw [black] (10) circle (2pt);
\filldraw [black] (11) circle (2pt);
\draw[-] (1) to (2);
\draw[-] (1) to (7);
\draw[-] (1) to (9);
\draw[-] (3) to (2);
\draw[-] (3) to (4);
\draw[-] (5) to (4);
\draw[-] (5) to (6);
\draw[-] (4) to (8);
\draw[-] (3) to (12);
\draw[-] (13) to (12);
\draw[-] (13) to (10);
\draw[-] (13) to (11);
\node[] at (1.325,-1.75) { $T_{2,2,4}^2$};
\end{tikzpicture}

\begin{tikzpicture}
\coordinate (1) at (-.35,.35);
\coordinate (2) at (0,0);
\coordinate (14) at (-.35,-.35);
\coordinate (3) at (.5,-0);
\coordinate (4) at (1.5,0);
\coordinate (5) at (2.5,0);
\coordinate (6) at (3,0);
\coordinate (7) at (3.35,.35);
\coordinate (18) at (3.35,-.35);
\coordinate (8) at (1.5,-.5);
\coordinate (16) at (2,-.5);
\coordinate (17) at (3,-.5);
\coordinate (9) at (1.5,-1);
\coordinate (10) at (2.5,-.5);
\coordinate (11) at (2.5,-1);
\coordinate (12) at (.5,-.5);
\coordinate (15) at (1,-.5);
\coordinate (13) at (.5,-1);
\filldraw [black] (1) circle (2pt);
\filldraw [black] (2) circle (2pt);
\filldraw [black] (3) circle (2pt);
\filldraw [black] (4) circle (2pt);
\filldraw [black] (5) circle (2pt);
\filldraw [black] (6) circle (2pt);
\filldraw [black] (7) circle (2pt);
\filldraw [black] (8) circle (2pt);
\filldraw [black] (9) circle (2pt);
\filldraw [black] (10) circle (2pt);
\filldraw [black] (11) circle (2pt);
\filldraw [black] (12) circle (2pt);
\filldraw [black] (13) circle (2pt);
\filldraw [black] (14) circle (2pt);
\filldraw [black] (15) circle (2pt);
\filldraw [black] (16) circle (2pt);
\filldraw [black] (17) circle (2pt);
\filldraw [black] (18) circle (2pt);
\draw[-] (1) to (2);
\draw[-] (14) to (2);
\draw[-] (3) to (2);
\draw[-] (3) to (4);
\draw[-] (5) to (4);
\draw[-] (5) to (6);
\draw[-] (7) to (6);
\draw[-] (4) to (8);
\draw[-] (8) to (9);
\draw[-] (5) to (10);
\draw[-] (10) to (11);
\draw[-] (10) to (17);
\draw[-] (6) to (18);
\draw[-] (3) to (12);
\draw[-] (13) to (12);
\draw[-] (15) to (12);
\draw[-] (16) to (8);
\node[] at (1.5,-1.4) { $y_3$};
\end{tikzpicture}\quad\quad
\begin{tikzpicture}
\coordinate (1) at (-.35,.35);
\coordinate (2) at (0,0);
\coordinate (14) at (-.35,-.35);
\coordinate (3) at (.5,-0);
\coordinate (4) at (1.5,0);
\coordinate (5) at (2.5,0);
\coordinate (19) at (.5,0.5);
\coordinate (20) at (1.5,0.5);
\coordinate (21) at (2.5,0.5);
\coordinate (6) at (3,0);
\coordinate (7) at (3.35,.35);
\coordinate (18) at (3.35,-.35);
\coordinate (8) at (1.5,-.5);
\coordinate (16) at (2,-.5);
\coordinate (17) at (3,-.5);
\coordinate (9) at (1.5,-1);
\coordinate (10) at (2.5,-.5);
\coordinate (11) at (2.5,-1);
\coordinate (12) at (.5,-.5);
\coordinate (15) at (1,-.5);
\coordinate (13) at (.5,-1);
\filldraw [black] (1) circle (2pt);
\filldraw [black] (2) circle (2pt);
\filldraw [black] (3) circle (2pt);
\filldraw [black] (4) circle (2pt);
\filldraw [black] (5) circle (2pt);
\filldraw [black] (6) circle (2pt);
\filldraw [black] (7) circle (2pt);
\filldraw [black] (8) circle (2pt);
\filldraw [black] (9) circle (2pt);
\filldraw [black] (10) circle (2pt);
\filldraw [black] (11) circle (2pt);
\filldraw [black] (12) circle (2pt);
\filldraw [black] (13) circle (2pt);
\filldraw [black] (14) circle (2pt);
\filldraw [black] (15) circle (2pt);
\filldraw [black] (16) circle (2pt);
\filldraw [black] (17) circle (2pt);
\filldraw [black] (18) circle (2pt);
\filldraw [black] (19) circle (2pt);
\filldraw [black] (20) circle (2pt);
\filldraw [black] (21) circle (2pt);
\draw[-] (1) to (2);
\draw[-] (14) to (2);
\draw[-] (3) to (2);
\draw[-] (3) to (19);
\draw[-] (3) to (4);
\draw[-] (20) to (4);
\draw[-] (5) to (4);
\draw[-] (5) to (21);
\draw[-] (5) to (6);
\draw[-] (7) to (6);
\draw[-] (4) to (8);
\draw[-] (8) to (9);
\draw[-] (5) to (10);
\draw[-] (10) to (11);
\draw[-] (10) to (17);
\draw[-] (6) to (18);
\draw[-] (3) to (12);
\draw[-] (13) to (12);
\draw[-] (15) to (12);
\draw[-] (16) to (8);
\node[] at (1.5,-1.4) { $Y_3$};
\end{tikzpicture}\quad\quad
\begin{tikzpicture}
\coordinate (1) at (-.35,.35);
\coordinate (2) at (0,0);
\coordinate (25) at (-.5,0);
\coordinate (14) at (-.35,-.35);
\coordinate (3) at (.5,-0);
\coordinate (4) at (1.5,0);
\coordinate (5) at (2.5,0);
\coordinate (19) at (.5,0.5);
\coordinate (20) at (1.5,0.5);
\coordinate (21) at (2.5,0.5);
\coordinate (6) at (3,0);
\coordinate (26) at (3.5,0);
\coordinate (7) at (3.35,.35);
\coordinate (18) at (3.35,-.35);
\coordinate (8) at (1.5,-.5);
\coordinate (15) at (.85,-.85);
\coordinate (16) at (1.85,-.85);
\coordinate (17) at (2.85,-.85);
\coordinate (22) at (.15,-.85);
\coordinate (23) at (1.15,-.85);
\coordinate (24) at (2.15,-.85);
\coordinate (9) at (1.5,-1);
\coordinate (10) at (2.5,-.5);
\coordinate (11) at (2.5,-1);
\coordinate (12) at (.5,-.5);
\coordinate (13) at (.5,-1);
\filldraw [black] (1) circle (2pt);
\filldraw [black] (2) circle (2pt);
\filldraw [black] (3) circle (2pt);
\filldraw [black] (4) circle (2pt);
\filldraw [black] (5) circle (2pt);
\filldraw [black] (6) circle (2pt);
\filldraw [black] (7) circle (2pt);
\filldraw [black] (8) circle (2pt);
\filldraw [black] (9) circle (2pt);
\filldraw [black] (10) circle (2pt);
\filldraw [black] (11) circle (2pt);
\filldraw [black] (12) circle (2pt);
\filldraw [black] (13) circle (2pt);
\filldraw [black] (14) circle (2pt);
\filldraw [black] (15) circle (2pt);
\filldraw [black] (16) circle (2pt);
\filldraw [black] (17) circle (2pt);
\filldraw [black] (18) circle (2pt);
\filldraw [black] (19) circle (2pt);
\filldraw [black] (20) circle (2pt);
\filldraw [black] (21) circle (2pt);
\filldraw [black] (22) circle (2pt);
\filldraw [black] (23) circle (2pt);
\filldraw [black] (24) circle (2pt);
\filldraw [black] (25) circle (2pt);
\filldraw [black] (26) circle (2pt);
\draw[-] (1) to (2);
\draw[-] (14) to (2);
\draw[-] (3) to (2);
\draw[-] (3) to (19);
\draw[-] (3) to (4);
\draw[-] (20) to (4);
\draw[-] (5) to (4);
\draw[-] (5) to (21);
\draw[-] (5) to (6);
\draw[-] (7) to (6);
\draw[-] (4) to (8);
\draw[-] (8) to (9);
\draw[-] (5) to (10);
\draw[-] (10) to (11);
\draw[-] (10) to (17);
\draw[-] (10) to (24);
\draw[-] (6) to (18);
\draw[-] (3) to (12);
\draw[-] (13) to (12);
\draw[-] (15) to (12);
\draw[-] (22) to (12);
\draw[-] (16) to (8);
\draw[-] (23) to (8);
\draw[-] (26) to (6);
\draw[-] (25) to (2);
\node[] at (1.5,-1.4) { $Z_3$};
\end{tikzpicture}
\end{center}
\end{example}

Unfortunately, only studying the combinatorial Laplacian spectra of trees is not sufficient to resolve Conjecture~\ref{conj:trees}, as there exist infinitely many pairs of nonisomorphic trees with the same combinatorial Laplacian spectrum, including pairs consisting of a caterpillar and a tree that is not a caterpillar \cite{Stanic}. An even stronger negative result due to McKay in \cite[Corollary 4.4]{McKay} states that the proportion of trees on $n$ vertices that are determined by their combinatorial Laplacian spectra goes to zero as $n$ goes to infinity. (Actually, McKay's result stated that the proportion goes to zero for trees determined by the adjacency spectra of their line graphs, which is equivalent, since for any tree $T$ the adjacency matrix of its line graph can be shown to equal to $K_T-2I$, where $K_T$ is the edge Laplacian of $T$ in a certain basis.)

However, Theorem~\ref{the:csf} does give insight into what information is captured by the chromatic symmetric function of a tree. It may also be worthwhile to study constructions of nonisomorphic trees with the same combinatorial Laplacian spectrum, either to search for counterexamples of the conjecture, or to understand the properties of possible counterexamples. Some relevant constructions can be found in \cite[Theorem 4.2]{McKay} and \cite[Theorem 3.1]{WHHL}.

\section{Acknowledgments}\label{sec:ack}  The authors would like to thank Will Evans and Joel Friedman for helpful conversations, Mike Huang for helping translate Chinese-language articles and Fei Wen for sending the authors a copy of \cite{WH}.

\bibliographystyle{plain}

\end{document}